\documentclass[11pt]{article}

\usepackage[letterpaper,margin=1in]{geometry}
\usepackage[T1]{fontenc}
\usepackage[utf8]{inputenc}
\usepackage{tgtermes}
\usepackage{newtxtext}
\usepackage{microtype}

\usepackage{amsmath,amsthm,mathtools}
\usepackage{newtxmath}
\usepackage{bm}
\usepackage{graphicx}
\usepackage{booktabs}
\usepackage{array}
\usepackage{xcolor}
\usepackage{algorithm}
\usepackage{algpseudocode}
\usepackage{tikz}
\usepackage[authoryear,round]{natbib}
\usepackage[colorlinks=true,linkcolor=blue,citecolor=blue,urlcolor=blue]{hyperref}

\allowdisplaybreaks
\theoremstyle{plain}
\newtheorem{theorem}{Theorem}
\newtheorem{lemma}{Lemma}
\newtheorem{proposition}{Proposition}

\newtheorem{assumption}{Assumption}

\theoremstyle{definition}

\theoremstyle{remark}
\newtheorem{remark}{Remark}

\newcounter{savedassumption}
\newenvironment{alphabetassumptions}{%
    \setcounter{savedassumption}{\value{assumption}}%
    \setcounter{assumption}{0}%
}{%
    \setcounter{assumption}{\value{savedassumption}}%
}

\newcommand{\Tr}{\operatorname{Tr}}

\title{High-Dimensional Interpolators Can Be Fragile:\\
Heavy Tails and High-Dimensional Large Deviations}

\author{%
Youheng Zhu\\
Department of Industrial Engineering and Management Sciences\\
Northwestern University\\
\texttt{youhengzhu@u.northwestern.edu}
\and
Yiping Lu\\
Department of Industrial Engineering and Management Sciences\\
Northwestern University\\
\texttt{yiping.lu@northwestern.edu}
}

\date{}

\begin{document}

\maketitle

\begin{abstract}
High-dimensional interpolation is common in modern machine learning, but its tail risk is less understood than its expected prediction risk. Existing theory shows that interpolating models can perform well in expectation, yet such guarantees do not determine the probability of rare, severe errors. In operations research and stochastic decision-making applications, rare estimation errors can have disproportionate downstream effects, so tail behavior matters alongside average performance.

We study the fragility of high-dimensional linear interpolators using large-deviation methods. We focus on ridgeless regression and compare it with ridge-regularized estimators. We first show that the risk of ridgeless regression can exhibit heavy-tailed behavior: although its expected risk may remain well controlled, its upper tail can decay much more slowly than that of regularized alternatives. We then quantify this phenomenon at the level of large-deviation rates. In the regime we study, ridge regularization suppresses fixed right-tail deviations at the \(n^2\) scale, whereas ridgeless regression has only \(n\log n\)-scale decay, where \(n\) is the sample size. This gap shows that interpolation can be statistically fragile even when it is accurate on average. Thus regularization affects the frequency of rare, high-impact risk events in addition to the usual bias-variance tradeoff.
\end{abstract}



\section{Introduction}

Recent work has sought to explain why predictors that interpolate the training data can still generalize well \citep{belkin2019reconciling,mallinar2022benign,zhang2021understanding}. This phenomenon runs counter to the classical view that interpolation should worsen out-of-sample performance. In high-dimensional regression, overparameterized estimators \citep{bartlett2020benign,belkin2018overfitting,belkin2021fit,cao2022benign,cheng2022memorize,hastie2022surprises,mei2022generalization}, including ridgeless and minimum-norm interpolators \citep{hastie2022surprises}, can achieve favorable expected prediction risk and, in some regimes, match or even outperform regularized alternatives. Thus, interpolation need not be harmful on average.

However, average risk is not the only criterion that matters. In many operations research, stochastic modeling, and data-driven decision-making settings, the reliability of an estimator depends not only on its expected performance but also on the probability of rare but severe failures \citep{duffie1997overview,artzner1999coherent}. An estimator with good average risk may still be fragile if its risk has a heavy upper tail. This fragility is especially relevant when predictions feed into downstream optimization, resource allocation, inventory control, pricing, queueing, or policy decisions~\citep{bertsimas2020predictive, elmachtoub2022smart, donti2017task}: rare estimation errors can be amplified by subsequent decisions and lead to substantial operational losses~\citep{smith2006optimizer}. This motivates the following question:

\textit{Does the risk of high-dimensional linear interpolators concentrate tightly around its typical value, or can it exhibit rare but extreme deviations? More broadly, can we characterize the risk distribution of interpolating estimators beyond expected risk?}

While existing theory has largely characterized expected risk and deterministic equivalents for high-dimensional ridge and ridgeless regression, recent distributional work has begun to go beyond mean behavior for these estimators \citep{han2026distribution}. However, what remains less understood is the tail behavior of the risk, especially the probability of fixed rare-risk events and the scale at which these probabilities decay.

We study this question at two complementary levels. First, we analyze fixed-dimensional tail asymptotics and identify the polynomial right tail generated by near-dependence, equivalently by one eigenvalue approaching the hard edge. We then turn to proportional asymptotics and use large-deviation methods to determine the decay rates of fixed rare-risk events \citep{dembo2009large,varadhan2010large}. This framework allows us to compare ridgeless interpolation and ridge regularization at the level of tail robustness, rather than only through average risk.

Our focus is high-dimensional \textit{ridgeless regression}, a canonical model of interpolation in the overparameterized regime. Ridgeless regression fits the training data without explicit regularization and is central to the modern theory of benign overfitting. We compare it with ridge regression, which introduces explicit regularization and therefore does not generally interpolate. This comparison isolates the role of regularization in controlling rare risk events.

The main results of this paper are summarized as follows.
Let \(R_n\) denote the ridgeless population risk and let \(r_\star\) be its
limiting predictive risk, defined in
Section~\ref{sec:prelim}. Existing high-dimensional results on asymptotic predictive risk for least squares and random-feature regression show that, in proportional regimes, risk is governed at first order by deterministic spectral limits predicted by the Marchenko--Pastur law \citep{dobriban2018high,hastie2022surprises,mei2022generalization}. Intuitively, the high-dimensional ridgeless risk fluctuates around the limiting predictive risk \(r_\star\) when \(p\) and \(n\) grow proportionally. In our notation, the risk decomposes exactly as
\[
    R_n=b_n+H_n,
    \qquad
    H_n=\frac1n\sum_i\lambda_{i,n}^{-1},
\]
where \(b_n\) is the deterministic bias and \(H_n\) is the random positive-spectrum inverse trace.
This separation is useful because all right-tail randomness studied below is carried by \(H_n\), whereas \(b_n\) only shifts the risk threshold. We obtain three results: a fixed-dimensional polynomial right tail, an $n\log n$ right-tail asymptotic, and a Gaussian left-tail LDP with speed $n^2$.

First, at fixed dimension, ridgeless risk has a genuine polynomial right tail.
Under the bounded-density
and nondegeneracy assumptions stated in Section~\ref{subsec:distributional-assumptions}, our Theorem~\ref{thm:fixed-dimensional-risk-tail} shows that
\[
    \mathbb P(R_n>b_n+x)=x^{-(|p-n|+1)/2+o(1)},
    \quad x\to\infty .
\] 
This shows that the population risk of ridgeless regression has polynomial right-tail decay.

This fixed-dimensional tail behavior raises a high-dimensional question: does the same hard-edge mechanism determine the large-deviation speed of a fixed right-tail deviation?

This is answered by our second result, which states that in the proportional regime \(p_n/n\to\gamma\in(0,\infty)\setminus\{1\}\),
ridgeless regression has only \(n\log n\)-scale large deviation right-tail decay. Under our
general-entry assumptions, for every fixed \(\delta>0\), Theorem~\ref{thm:ridgeless-right-tail} shows that
\[
    \lim_{n\to\infty}
    \frac1{n\log n}
    \log\mathbb P(R_n>r_\star+\delta)
    =
    -\frac{|1-\gamma|}{2}.
\]
The constant is independent of \(\delta\). The mechanism of this rate is the
existence of a single positive eigenvalue reaching the microscopic scale
\(n^{-1}\), which directly gives the lower bound. The matching upper bound
proves that no cheaper mechanism exists. It separates eigenvalues above and
below a shrinking hard-edge cutoff: the nonmicroscopic contribution is
controlled by a bounded-Lipschitz spectral statistic, while the
microscopic contribution is controlled by a hard-edge counting estimate for the
\(k\)-th ordered eigenvalue. This estimate is proved using projection
small-ball estimates~\citep{nguyen2018random} and a Grassmannian covering
argument, and a weighted pigeonhole step identifies a single microscopic
eigenvalue as the leading mechanism.

Our third result shows that the Gaussian left tail remains an \(n^2\)-scale spectral large-deviation
event. For Gaussian entries, Theorem~\ref{thm:ridgeless-left-tail} shows that
\[
    \lim_{n\to\infty}\frac1{n^2}\log\mathbb P(R_n\le x)
    =
    -\mathcal J_{\kappa_\gamma,\rho_\gamma}
    \left(\frac{x-b_\gamma}{\kappa_\gamma}\right),
\]
where \(\mathcal J\) is the left-tail contraction rate induced by the
inverse-moment functional, as defined in Section~\ref{sec:gaussian-left-tail}, Theorem~\ref{thm:ridgeless-left-tail}. In the nontrivial regime
\(b_\gamma<x<r_\star\) where \(b_n\to b_\gamma\), this gives an \(n^2\)-speed left-tail decay. Thus the typical risk separates two
different mechanisms: the left tail is produced by a collective deformation of
the empirical spectral distribution, while the right tail is produced by a
single hard-edge eigenvalue.

This behavior contrasts sharply with ridge regression. For every fixed
\(\lambda>0\), the ridge risk is a bounded continuous spectral functional. Hence
in the Gaussian case the Wishart spectral large deviation principle and the contraction principle yield an LDP for ridge risk with speed \(n^2\), and under our general-entry assumptions the
ridge right tail is suppressed at the \(n^2\) scale, as shown in Theorem~\ref{thm:gaussian-ridge-risk-ldp} and Proposition~\ref{prop:general-ridge-right-tail-upper}. Therefore fixed right-tail
events satisfy
\[
    \mathbb P(R_{n,\lambda}>r_{\lambda,\star}+\delta)
    =
    \exp\{-\Omega(n^2)\},
    \qquad
    \mathbb P(R_n>r_\star+\delta)
    =
    \exp\{-\Theta(n\log n)\}.
\]
This gap in the large deviation speed quantifies the tail fragility created by removing ridge regularization.

We also show that the lower-bound mechanism is not specific to the minimum-norm
choice. In the overparameterized regime, any exact linear interpolator must fit
the training noise in the row space of the design. This forces an
algorithm-independent inverse-trace lower bound and yields both a
fixed-dimensional polynomial lower bound and a proportional \(n\log n\)-scale
lower bound for arbitrary exact linear interpolating algorithms. The matching
upper bounds remain specific to the minimum-norm estimator.

\paragraph{Roadmap.}
The rest of the paper proves these statements. Section~\ref{sec:prelim}
introduces the regression model, notation for the random matrix and its spectrum,
the distributional assumptions, and the Wishart spectral large deviation principle. 
Section~\ref{sec:fixed-dimensional-tail} proves the fixed-dimensional
polynomial heavy-tail theorem. Section~\ref{sec:proportional-right-tail} proves
the sharp \(n\log n\)-scale right-tail asymptotic by matching a
single-eigenvalue lower bound with an upper bound excluding the combined
nonmicroscopic contribution and multi-eigenvalue hard-edge mechanisms.

Section~\ref{sec:gaussian-left-tail}
proves the Gaussian \(n^2\)-scale left-tail result. 
Section~\ref{sec:ridge-right-tail} records the \(n^2\)-scale
right-tail suppression for ridge regression, in contrast to the results in
Section~\ref{sec:proportional-right-tail}.
Section~\ref{sec:beyond-ridgeless} extends the lower-bound mechanism beyond
ridgeless regression to arbitrary exact linear interpolating algorithms.
Section~\ref{sec:numerics} presents numerical experiments. The proofs of
auxiliary estimates are collected in
Section~\ref{sec:proofs-auxiliary-estimates}.

These findings complement the recent theory of high-dimensional interpolators: expectation-based results show that interpolation can be benign in an average sense, while our results show that this benign behavior need not extend to tail risk.

\section{Preliminaries and notation}
\label{sec:prelim}

This section sets up the regression model and the random matrix notation used throughout the paper. The main point is to reduce the population risk of ridgeless regression to a singular spectral statistic of the sample covariance matrix. In this representation, the Marchenko--Pastur law gives the deterministic first-order prediction for the risk, while the large-deviation questions studied in this paper concern the probability that the realized risk deviates from this prediction.

\subsection{Ridge and ridgeless regression}

We consider the linear model $y=X\beta+\varepsilon$,
where \(X\in \mathbb R^{n\times p}\) is the design matrix, \(\varepsilon\sim N(0,I_n)\) is independent noise, and the entries of \(X\) are independent copies of a mean-zero, variance-one random variable \(\xi\). Following \citet{dobriban2018high} and \citet{hastie2022surprises}, we work in the high-dimensional proportional asymptotic regime
\(\gamma_n:=\frac{p}{n}\to \gamma\in (0,\infty)\setminus\{1\}.\) The case \(\gamma<1\) is underparameterized, while \(\gamma>1\) is overparameterized.

We also adopt the dense random effects convention of \cite{dobriban2018high}, where we assume the coefficient $\beta$ is sampled from a distribution with $\mathbb E\beta=0,
\mathbb E[\beta\beta^\top]=\frac{\alpha^2}{p}I_p.$
Here \(\alpha^2=\mathbb E\|\beta\|_2^2\) is the signal strength, while the noise variance is normalized to one. This convention keeps the signal norm of order one as \(p\) grows and makes the conditional prediction risk a function only of the realized design matrix. It also isolates the spectral source of risk fluctuations. In the overparameterized regime, a setting with fixed coefficient sequence \(\beta_n\) would introduce the projection bias \(\|(I-P_{\operatorname{row}(X)})\beta_n\|_2^2\); the random effects risk replaces this by its isotropic average \(\alpha^2(1-\gamma_n^{-1})\). This replacement affects only the bounded and converging bias contribution and does not change the inverse trace mechanism responsible for the tail behavior below.

For \(\lambda>0\), the ridge estimator is $\widehat\beta_\lambda
=
\arg\min_{b\in\mathbb R^p}
\{\|Xb-y\|_2^2+n\lambda\|b\|_2^2\}
=
(X^\top X+n\lambda I_p)^{-1}X^\top y .$ The ridgeless estimator is obtained by removing the explicit penalty:
\[
\widehat\beta_0
=
(X^\top X)^+X^\top y
=
\lim_{\lambda\downarrow 0}\widehat\beta_\lambda .
\]
It is the minimum Euclidean-norm least-squares solution. When \(p>n\), this estimator interpolates the training data almost surely under continuous entry distributions.

Let \(x_0\) be an independent test point with \(\mathbb E x_0x_0^\top=I_p\). We define the conditional population risk $R_{n,\lambda}(X)
:=
\mathbb E\left[
\{x_0^\top(\widehat\beta_\lambda-\beta)\}^2
\mid X
\right],$ where the expectation is over \(\beta,\varepsilon\), and \(x_0\), conditional on \(X\). This definition compares the fitted linear predictor with the noiseless population signal, so the irreducible test noise is not included. Writing \(\widehat\Sigma:=\frac{1}{n}X^\top X\), a standard bias-variance calculation gives
\begin{equation}
\label{eq:ridge-risk-prelim}
R_{n,\lambda}(X)
=
\frac{\alpha^2\lambda^2}{p}
\operatorname{Tr}(\widehat\Sigma+\lambda I_p)^{-2}
+
\frac{1}{n}
\operatorname{Tr}\!\left[
\widehat\Sigma(\widehat\Sigma+\lambda I_p)^{-2}
\right].
\end{equation}
Equivalently,
\begin{equation}
\label{eq:ridge-risk-equivalent-prelim}
R_{n,\lambda}(X)
=
\frac{\gamma_n}{p}
\operatorname{Tr}(\widehat\Sigma+\lambda I_p)^{-1}
+
\frac{\lambda(\lambda\alpha^2-\gamma_n)}{p}
\operatorname{Tr}(\widehat\Sigma+\lambda I_p)^{-2}.
\end{equation}

Taking \(\lambda\downarrow 0\) gives the ridgeless risk. If \(p<n\), then \(\widehat\Sigma\) is invertible almost surely and
\(
R_n(X)=\frac{1}{n}\operatorname{Tr}\widehat\Sigma^{-1}.
\)
If \(p>n\), then \(X^\top X/n\) has \(p-n\) zero eigenvalues, and the positive spectrum is instead represented by \(XX^\top/n\). In this case,
\(
R_n(X)
=
\alpha^2\left(1-\frac{1}{\gamma_n}\right)
+
\frac{1}{n}\operatorname{Tr}
\left(\frac{1}{n}XX^\top\right)^{-1}.\) Thus in both regimes, if \(0<\lambda_1\le \cdots\le \lambda_q\), \(q=\min\{n,p\}\), are the positive eigenvalues of the relevant normalized covariance matrix, then
\begin{equation}
\label{eq:ridgeless-risk-prelim}
R_n(X)
=
b_n+
\frac{1}{n}\sum_{i=1}^q \lambda_i^{-1},
\qquad
b_n=
\begin{cases}
0, & p<n,\\
\alpha^2(1-\gamma_n^{-1}), & p>n.
\end{cases}
\end{equation}
Equation~\eqref{eq:ridgeless-risk-prelim} reduces the random part of ridgeless risk to an inverse spectral moment. Ridge regression depends on the regularized spectral functions \((u+\lambda)^{-1}\) and \((u+\lambda)^{-2}\), which remain bounded near zero for fixed \(\lambda>0\). Ridgeless regression instead depends on the singular function \(u^{-1}\). Consequently, small positive eigenvalues can dominate the upper tail of the risk.

\subsection{Spectrum notation and the Marchenko--Pastur law}
\label{sec:random-matrix-4-ridgeless}

We now introduce the random matrix notation independently of the regression model. Let \(G_n\in\mathbb R^{N_n\times q_n}\), \(N_n\ge q_n\), have independent entries with mean zero and variance one. Define
\begin{equation}
\label{eq:define-rn}
r_n:=N_n-q_n,\qquad
A_n:=\frac{1}{n}G_n^\top G_n,\qquad
0<\lambda_{1,n}\le \cdots\le \lambda_{q_n,n},\qquad
\widehat\mu_n
:=
\frac{1}{q_n}
\sum_{i=1}^{q_n}\delta_{\lambda_{i,n}} .
\end{equation}
Here \(\lambda_{i,n}\) are the eigenvalues of \(A_n\), and
\(\widehat\mu_n\) is the empirical positive spectral distribution.
We assume the proportional rectangular asymptotic regime
\(\frac{q_n}{n}\to \kappa\in(0,\infty),\frac{r_n}{n}\to \rho\in[0,\infty).\)
Equivalently, \(N_n/n\to \kappa+\rho\). Under standard moment conditions, \(\widehat\mu_n\) converges weakly to the Marchenko--Pastur law \(\mu_{\kappa,\rho}\) almost surely \citep{marvcenko1967distribution, bai2010spectral, pillai2014universality}, supported on
\(
[\ell_{\kappa,\rho},L_{\kappa,\rho}],\) where \(
\ell_{\kappa,\rho}:=(\sqrt{\kappa+\rho}-\sqrt{\kappa})^2,
L_{\kappa,\rho}:=(\sqrt{\kappa+\rho}+\sqrt{\kappa})^2,\)
with density
\[
d\mu_{\kappa,\rho}(u)
=
\frac{1}{2\pi\kappa u}
\sqrt{(L_{\kappa,\rho}-u)(u-\ell_{\kappa,\rho})}
\mathbf 1_{[\ell_{\kappa,\rho},L_{\kappa,\rho}]}(u)\,du .
\]

The central spectral statistic considered in this paper is
\(
H_n
:=
\frac{1}{n}
\sum_{i=1}^{q_n}\lambda_{i,n}^{-1}
=
\frac{q_n}{n}
\int u^{-1}\,d\widehat\mu_n(u).\) We refer to \(H_n\) as the positive spectrum inverse trace, or the inverse moment of the empirical spectral measure \(\widehat\mu_n\) scaled by \(q_n/n\). When \(\rho>0\), the lower edge \(\ell_{\kappa,\rho}\) is strictly positive, and the Marchenko--Pastur prediction for \(H_n\) is finite:
\(
h_{\kappa,\rho}
:=
\kappa\int u^{-1}\,d\mu_{\kappa,\rho}(u).
\)
Equivalently, \(h_{\kappa,\rho}=\kappa/\rho\), but we keep the integral notation because the proofs compare \(H_n\) with spectral functionals of \(\widehat\mu_n\).

The ridgeless regression model is obtained from this notation by taking the positive spectrum of the sample covariance matrix. If \(X_n\in\mathbb R^{n\times p_n}\) and \(p_n/n\to\gamma\in(0,\infty)\setminus\{1\}\), then
\begin{equation}
\label{eq:define-kappa_gamma}
N_n=\max\{n,p_n\},
q_n=\min\{n,p_n\},
r_n=|n-p_n|.
\end{equation}
Thus
\(
\kappa_\gamma:=\min\{1,\gamma\},
\rho_\gamma:=|1-\gamma|.\) Since \(\gamma\ne 1\), we have \(\rho_\gamma>0\), so the limiting lower spectral edge is positive. Combining the exact risk representation~\eqref{eq:ridgeless-risk-prelim} with the Marchenko--Pastur limit gives the limiting predictive risk
\begin{equation}
\label{eq:typical-risk-prelim}
r_\star
=
b_\gamma+h_{\kappa_\gamma,\rho_\gamma},\qquad b_\gamma
:=
\alpha^2(1-\gamma^{-1})_+ .
\end{equation}
The value \(r_\star\) is the deterministic first-order prediction for ridgeless risk in the proportional regime. The large-deviation results below ask how likely it is for the realized risk \(R_n=b_n+H_n\) to exceed or fall below this Marchenko--Pastur prediction.

\subsection{Gaussian Wishart ensembles and spectral large deviations}
\label{sec:Gaussian Wishart ensembles and spectral large deviations}

For Gaussian designs, the positive spectrum has an explicit joint density. Suppose \(G\in\mathbb R^{N\times q}\), \(N\ge q\), has independent \(N(0,1)\) entries, and set \(r=N-q\). For
\(
A=\frac{1}{n}G^\top G ,
\)
let \(0<\lambda_1\le \cdots\le \lambda_q\) be its eigenvalues. The unordered eigenvalue density is the Laguerre orthogonal ensemble density~\citep{james1964distributions, hiai1998eigenvalue, forrester2010log}
\begin{equation}
\label{eq:wishart-density-prelim}
p(\lambda_1,\ldots,\lambda_q)
=
\frac{1}{Z_{N,q,n}}
\prod_{i=1}^q
\lambda_i^{(r-1)/2}
e^{-n\lambda_i/2}
\prod_{1\le i<j\le q}
|\lambda_i-\lambda_j|,
\qquad
\lambda_i>0 .
\end{equation}
The factor \(\lambda_i^{(r-1)/2}\) describes the hard edge repulsion from zero, a classical feature of Gaussian Wishart spectra \citep{james1964distributions,edelman1988eigenvalues,forrester2010log}. This term is the finite-dimensional source of the polynomial small eigenvalue behavior and, in the proportional regime, of the \(n\log n\) cost of moving one positive eigenvalue to the microscopic hard edge. Related results for the smallest Wishart--Laguerre eigenvalues have been studied extensively in random matrix theory \citep{tao2010random,katzav2010large}.

We also use the Gaussian Wishart spectral large deviation principle. Suppose
\(
\frac{q}{n}\to\kappa\in(0,\infty),
\frac{r}{n}\to\rho>0.
\) Then the empirical spectral distribution \(\widehat\mu\) satisfies a large deviation principle on \(\mathcal P(\mathbb R_+)\), endowed with the weak topology, with speed \(n^2\) and good rate function~\citep{hiai2000semicircle}
\begin{equation}
\label{eq:wishart-rate-prelim}
I_{\kappa,\rho}(\mu)
=
\frac{\kappa}{2}
\int (u-\rho\log u)\,d\mu(u)
-
\frac{\kappa^2}{2}
\iint \log|u-v|\,d\mu(u)d\mu(v)
-
C_{\kappa,\rho}.
\end{equation}
The constant \(C_{\kappa,\rho}\) is chosen so that \(\inf_{\mu\in\mathcal P(\mathbb R_+)} I_{\kappa,\rho}(\mu)=0 .\) With this normalization, the unique minimizer of \(I_{\kappa,\rho}\) is the Marchenko--Pastur law \(\mu_{\kappa,\rho}\). Thus the spectral large deviation principle refines the Marchenko--Pastur law: the Marchenko--Pastur measure gives the typical empirical spectrum, while \(I_{\kappa,\rho}\) quantifies the \(n^2\) scale cost of a collective deformation of the entire spectrum.


This spectral LDP will be used only for deviations that are visible at the level of the empirical spectral measure. In particular, for Gaussian designs, the left tail of the inverse-trace functional is governed by a macroscopic deformation of the empirical spectrum and has the natural \(n^2\)-scale variational rate suggested by the above LDP through the contraction principle. By contrast, the right tail is governed by the singularity of \(u^{-1}\) at the hard edge: a single eigenvalue at the microscopic scale \(n^{-1}\) changes the inverse trace by order one but does not change \(\widehat\mu\) at the same order in the weak topology. The right-tail analysis therefore requires separate hard-edge estimates rather than the \(n^2\)-speed spectral LDP.

\subsection{Distributional assumptions}
\label{subsec:distributional-assumptions}

The assumptions below reflect the two probabilistic mechanisms analyzed in this
paper. Assumption~\ref{ass:gaussian-entries} isolates the Gaussian Wishart
setting, where the exact eigenvalue density and the spectral large deviation
principle are available. This structure is used for the left tail, which is
driven by a collective deformation of the empirical spectrum. In contrast,
Assumption~\ref{ass:general-entries} covers the fixed-dimensional heavy tail and
the proportional right-tail results for general entry distributions. It is broad
enough for continuous, normalized, light-tailed covariates, as commonly arise
after preprocessing in high-dimensional statistical and data-driven decision
models. It includes Gaussian and other continuous subgaussian designs with
nondegenerate local density and log-Sobolev regularity, while excluding discrete or locally degenerate
designs for which the required small-ball and hard-edge estimates may fail.

Methodologically, Assumption~\ref{ass:general-entries} retains the Gaussian-type
regularity needed away from the hard edge: the log-Sobolev condition gives
uniform norm control and dimension-free Lipschitz concentration for bounded
spectral statistics. What is not retained is the exact Wishart eigenvalue
density. The replacement is a small-ball structure at the hard edge: bounded
density gives upper small-ball control, and nondegenerate local mass near the
origin gives the matching lower mass needed for the one-column near-dependence
mechanism.

\begin{alphabetassumptions}

\begin{assumption}[Gaussian entries]
\label{ass:gaussian-entries}
The entries of the relevant random matrix are i.i.d. \(N(0,1)\).
\end{assumption}

\begin{assumption}[General entries]
\label{ass:general-entries}
The entries of the relevant random matrix are i.i.d. copies of a real random variable \(\xi\) satisfying the following conditions:
\begin{enumerate}
    \renewcommand{\labelenumi}{\textup{B.\arabic{enumi}.}}
    \renewcommand{\theenumi}{B.\arabic{enumi}}

    \item \textbf{Normalization.} \(\mathbb E\xi=0\) and \(\mathbb E\xi^2=1\).

    \item \textbf{Bounded density.} The law of \(\xi\) has a density \(p_\xi\) satisfying
    \(\|p_\xi\|_\infty\le K_{\rm up}<\infty\).

    \item \textbf{Nondegeneracy near zero.} There exist \(u_0>0\) and \(k_{\rm low}>0\) such that
    \(p_\xi(u)\ge k_{\rm low}\) for all \(|u|\le u_0\).

    \item \textbf{Log-Sobolev inequality.} The law of \(\xi\) satisfies a logarithmic Sobolev
    inequality with constant \(C_{\rm LS}<\infty\).
\end{enumerate}
\end{assumption}

\end{alphabetassumptions}

\section{Fixed-dimensional polynomial right tail}
\label{sec:fixed-dimensional-tail}

In this section, we work in the fixed-dimensional regime: \(N\ge q\) are fixed, \(r:=N-q\),
and the tail level \(x\) tends to infinity.
Constants in this section may depend on \(N\), \(q\), and the density
parameters, but not on \(x\).

Let \(G\in\mathbb R^{N\times q}\) have i.i.d. entries satisfying the bounded-density
and nondegeneracy-near-zero conditions B.2--B.3 in
Assumption~\ref{ass:general-entries}. Set $W:=G^\top G$.
Since the entries have densities, \(G\) has full column rank almost surely.
In the ridgeless regression application one takes
\(N=\max\{n,p\}\), \(q=\min\{n,p\}\), and the risk differs from
\(\Tr(W^{-1})\) only by the deterministic contribution \(b\).

\subsection{The power-law tail of ridgeless risk}

At fixed dimension, a nearly dependent column creates a small eigenvalue and hence a large inverse-trace contribution. Theorem~\ref{thm:fixed-dimensional-risk-tail} quantifies the resulting polynomial right tail and its exponent \((r+1)/2\).

\begin{theorem}[Fixed-dimensional heavy tail of ridgeless risk]
\label{thm:fixed-dimensional-risk-tail}
\label{thm:fixed-dimensional-ridgeless-heavy-tail}
Fix the sample size and dimension. Suppose the entries in the ridgeless
regression design satisfy the bounded-density condition B.2 and the
nondegeneracy-near-zero condition B.3 in
Assumption~\ref{ass:general-entries}. In the
positive-spectrum representation of \eqref{eq:ridgeless-risk-prelim}, write
    $R=b+\Tr\{(G^\top G)^{-1}\}$
where \(b\) is the deterministic bias in the ridgeless risk formula.

Then
there exist constants \(0<c<C<\infty\) and \(x_0<\infty\) such that, for all
\(x\ge x_0\),
\[
    c x^{-(r+1)/2}
    \le
    \mathbb P(R>b+x)
    \le
    C x^{-(r+1)/2}.
\]
Equivalently, we have $ \lim_{x\to\infty}
    \frac{\log \mathbb P(R>b+x)}{\log x}
    =
    -\frac{r+1}{2}.$
\end{theorem}

In the Gaussian case, this exponent is consistent with the classical moment calculations for inverse Wishart and inverse Laguerre matrices: a tail of order \(x^{-(r+1)/2}\) implies that moments of \(\operatorname{Tr}(W^{-1})\) exist only below the corresponding threshold, matching the moment-threshold behavior found in \citet{matsumoto2012general,kumari2017moments}. Our contribution is a two-sided tail estimate under general bounded-density, locally nondegenerate entries, rather than only for the Gaussian Wishart ensemble.

\subsection{Proof of Theorem~\ref{thm:fixed-dimensional-risk-tail}}\label{sec:proof-thm1}

Let \(Y:=\Tr((G^\top G)^{-1})\). By the fixed-dimensional ridgeless risk
representation, $R
    =
    b+
    \frac1n\Tr\left((n^{-1}G^\top G)^{-1}\right)
    =
    b+Y.$ It therefore suffices to prove the following tail bound for the inverse-trace statistic $Y$:
\begin{equation}
    c x^{-(r+1)/2}
    \le
    \mathbb P(Y>x)
    \le
    C x^{-(r+1)/2},
    \qquad x\ge x_0,
    \label{eq:fixed-dimensional-inverse-tail-bounds}
\end{equation}
for some constants \(0<c<C<\infty\) and \(x_0<\infty\).

We now prove \eqref{eq:fixed-dimensional-inverse-tail-bounds}. We use the negative second-moment identity and upper and lower small-ball bounds for column-to-subspace distances; the auxiliary proofs are deferred to Section~\ref{sec:proofs-auxiliary-estimates}.
We first rewrite the inverse trace \(Y:=\Tr(W^{-1})\) as a sum of
inverse squared distances from one column to the span of the others~\citep{tao2010random2}.

\begin{lemma}[Negative second moment identity]
\label{lem:negative-second-moment-fixed}
Let \(G=[g_1,\ldots,g_q]\in\mathbb R^{N\times q}\) have full column rank, and
let \(W=G^\top G\). Put
\(H_i:=\operatorname{span}\{g_j:j\neq i\}\) and
\(D_i:=\operatorname{dist}(g_i,H_i)\). Then
\[
    \Tr(W^{-1})
    =
    \sum_{i=1}^q \frac1{D_i^2}.
\]
\end{lemma}

By Lemma~\ref{lem:negative-second-moment-fixed},
\(Y=\sum_{i=1}^q D_i^{-2}\).
Since \(H_i\) is generated by all columns except \(g_i\), it is independent of
\(g_i\). Also, almost surely \(\dim(H_i)=q-1\), so \(H_i\) has codimension
\(N-(q-1)=r+1\). Thus the tail of \(Y\) reduces to small-ball probabilities for the distance between a random column and the span of the remaining columns. We prove the upper bound and the lower bound separately below.

\paragraph{Upper bound for the right tail.} We obtain the upper bound for the right tail via an anti-concentration analysis of the subspace projection of a random vector.

If \(Y>y\), then at least one term in the sum exceeds
\(y/q\). Hence
    $\{Y>y\}
    \subseteq
    \bigcup_{i=1}^q
    \left\{
        D_i^2 < \frac{q}{y}
    \right\}$.
We now use the fixed-dimensional subspace distance upper bound for the distance
from an independent column to a fixed subspace. This is shown by the following
Proposition~\ref{cor:distance-subspace-upper-small-ball}, which is a direct
corollary of the projection small-ball estimate of
\citet[Theorem~2.5]{nguyen2018random} under the bounded-density condition in
Assumption~\ref{ass:general-entries}.

\begin{proposition}[Upper small-ball bound for distance to a subspace]
\label{cor:distance-subspace-upper-small-ball}
Assume the bounded-density condition B.2 in
Assumption~\ref{ass:general-entries}.
 Let \(Z=(\xi_1,\ldots,\xi_N)\) have
independent coordinates satisfying this condition, and let
\(H\subseteq\mathbb R^N\) have codimension \(d\ge1\). Then, for every \(t>0\),
\[
    \mathbb P(\operatorname{dist}(Z,H)\le t)
    \le
    \left(
        \frac{C K_{\rm up}t}{\sqrt d}
    \right)^d .
\]
In particular, for fixed \(N,d\), the probability is at most
\(C_{N,d}t^d\).
\end{proposition}

Conditioning on \(H_i\) and applying
Proposition~\ref{cor:distance-subspace-upper-small-ball} with \(d=r+1\) gives
\[
\begin{aligned}
    \mathbb P(Y>y)
    &\le
    \sum_{i=1}^q
    \mathbb P\left(D_i<\sqrt{\frac{q}{y}}\right) \le
    C y^{-(r+1)/2}
\end{aligned}
\]
for all \(y>0\), after changing the constant. Thus the upper bound is proved.



\paragraph{Lower bound for the right tail.}
For the reverse inequality, we use the inclusion
\(\{D_1^2<1/y\}\subseteq\{Y>y\}\). Thus it is enough to lower bound the
probability that one column lies very close to the span of the remaining
columns. We use the matching lower small-ball bound stated in
Proposition~\ref{cor:distance-subspace-lower-small-ball} below.

For the lower bound, restrict the \(H^\perp\)-component of \(X\) to a
\(d\)-dimensional ball and the \(H\)-component to a central cube section.
Vaaler's theorem~\citep{rogers1958packing,vaaler1979geometric} controls the
section volume, and the density lower bound converts it into probability; the
remaining \(a^d\) factor is the ball volume.

\begin{proposition}[Lower small-ball bound for distance to a subspace]
\label{cor:distance-subspace-lower-small-ball}
\label{prop:lower-small-ball-distance-subspace}
Assume the nondegeneracy-near-zero condition B.3 in
Assumption~\ref{ass:general-entries}, with constants
\(u_0,k_{\rm low}>0\). 
Let \(X=(\xi_1,\ldots,\xi_N)\) have i.i.d. coordinates satisfying this
condition, and let \(H\subseteq\mathbb R^N\) have codimension \(d\). Then, for
every \(0<a\le u_0/2\),
\[
    \mathbb P(\operatorname{dist}(X,H)\le a)
    \ge
    k_{\rm low}^N u_0^{N-d}\operatorname{Vol}_d(B_2^d(a)),
\]
where \(B_2^d(a)\) denotes the radius-\(a\) \(d\)-dimensional Euclidean ball and \(\operatorname{Vol}_d\) denotes
\(d\)-dimensional Lebesgue measure. In particular, for fixed \(N,d\), this is bounded below by
\(c_{\rm low}a^d\) for \(0<a\le u_0/2\).
\end{proposition}

Conditioning on \(H_1\) and applying
Proposition~\ref{cor:distance-subspace-lower-small-ball} with \(d=r+1\), we obtain,
for all sufficiently large \(y\),
\[
    \mathbb P(Y>y)
    \ge
    \mathbb P\left(D_1<y^{-1/2}\right)
    \ge
    c y^{-(r+1)/2}.
\]
Then the lower bound is proved. Taking \(y=x\) proves
    $c x^{-(r+1)/2}
    \le
    \mathbb P(Y>x)
    \le
    C x^{-(r+1)/2}$
for all sufficiently large \(x\). This is exactly
\eqref{eq:fixed-dimensional-inverse-tail-bounds}, and hence proves
Theorem~\ref{thm:fixed-dimensional-risk-tail}.

\section{High-dimensional large deviations of the right tail}
\label{sec:proportional-right-tail}

We now move from fixed-dimensional tail asymptotics to the proportional
high-dimensional regime. The Marchenko--Pastur law gives the first-order
spectral prediction for \(H_n\). Throughout this section, we assume the
proportional asymptotic model
\[
    \frac{q_n}{n}\to\kappa\in(0,\infty),
    \qquad
    \frac{r_n}{n}\to\rho\in(0,\infty).
\]
We use the positive spectrum notation from Section~\ref{sec:prelim}. In
particular, \(H_n=n^{-1}\sum_{i=1}^{q_n}\lambda_{i,n}^{-1}\), and
\(h_{\kappa,\rho}\) denotes its Marchenko\textendash Pastur asymptotic limit.
We study the decay of \(\mathbb P(H_n>h_{\kappa,\rho}+\delta)\) for fixed
\(\delta>0\).

\subsection{Large deviation right-tail of ridgeless regression risk}
For fixed \(\delta>0\), the event \(H_n>h_{\kappa,\rho}+\delta\) is driven by one positive eigenvalue at scale \(n^{-1}\), which contributes order one through the singular weight \(u^{-1}\). This spectral mechanism will then be translated back to the ridgeless regression risk. In the ridgeless specialization, \(R_n=b_n+H_n\), and the remaining term \(b_n\) is deterministic and converges. Thus the fixed right-tail event for \(R_n\) is controlled, at the \(n\log n\) scale, by the event that the positive spectrum develops a microscopic hard-edge eigenvalue. In this specialization, the exponent is determined by \(\rho_\gamma=\lvert1-\gamma\rvert\). The next theorem states the corresponding right tail large deviation principle for ridgeless regression.

\begin{theorem}[Right-tail large deviations for ridgeless regression]
\label{thm:ridgeless-right-tail}
Assume the normalization, bounded-density, nondegeneracy-near-zero, and
log-Sobolev conditions B.1, B.2, B.3, and B.4 in
Assumption~\ref{ass:general-entries}.
In the ridgeless regression specialization with
\(p_n/n\to\gamma\in(0,\infty)\setminus\{1\}\), write
\(\rho_\gamma=|1-\gamma|\). Then, for every fixed \(\delta>0\),
\[
    \lim_{n\to\infty}
    \frac{1}{n\log n}
    \log \mathbb P(R_n>r_\star+\delta)
    =
    -\frac{\rho_\gamma}{2}.
\]
\end{theorem}

Equivalently, the theorem shows \(\mathbb P(R_n>r_\star+\delta) =\exp\left\{
        -\left(\frac{\rho_\gamma}{2}+o(1)\right)n\log n
    \right\}.\)
This characterizes the sharp
logarithmic asymptotic for fixed large deviation right-tail events. 

The mechanism as shown in the proof
is a single eigenvalue at the hard edge: an event sending \(\lambda_{1,n}\) to the microscopic scale
\(n^{-1}\) already creates an \(O(1)\) upward deviation of the inverse trace. Additionally, the upper bound shows that alternative mechanisms, in which eigenvalues away from the microscopic hard edge make an essential contribution, are negligible on the \(n\log n\) scale; their probabilities are \(\exp{-\omega(n\log n)}\).

\subsection{Proof of Theorem~\ref{thm:ridgeless-right-tail}}\label{sec:proof-thm2}

It suffices to prove the sharp fixed-level right-tail asymptotic for the
inverse-trace statistic:
\begin{equation}
    \lim_{n\to\infty}
    \frac{1}{n\log n}
    \log \mathbb P(H_n>h_{\kappa,\rho}+\delta)
    =
    -\frac{\rho}{2},
    \qquad \delta>0.
    \label{eq:proportional-right-tail-asymptotic}
\end{equation}
Indeed, applied with
\((\kappa,\rho)=(\kappa_\gamma,\rho_\gamma)\), this matrix estimate transfers
directly to the risk statement.
In the regression specialization, \(R_n=b_n+H_n\), \(b_n\to b_\gamma\), and
\(r_\star=b_\gamma+h_{\kappa_\gamma,\rho_\gamma}\). Hence, for any fixed
\(\delta>0\), the events with thresholds \(r_\star+\delta\) are squeezed
between the corresponding spectral events with thresholds
\(h_{\kappa_\gamma,\rho_\gamma}+\delta/2\) and
\(h_{\kappa_\gamma,\rho_\gamma}+2\delta\), for all sufficiently large \(n\).
The matrix asymptotic \eqref{eq:proportional-right-tail-asymptotic} then gives
the same exponent for both spectral thresholds.

Therefore, it remains to prove \eqref{eq:proportional-right-tail-asymptotic}. We do this by
separating the two directions. The lower bound considers a single microscopic
eigenvalue event, which already forces a fixed upward deviation of \(H_n\). The
upper bound separates the inverse trace into a nonmicroscopic contribution,
containing all eigenvalues above the shrinking cutoff \(\eta_n\), and a
microscopic hard-edge contribution below \(\eta_n\). The nonmicroscopic part is
then controlled by the concentration of a bounded Lipschitz spectral statistic, while the
microscopic part is controlled by the hard-edge counting argument.

\subsubsection{Lower bound: one microscopic eigenvalue}
\label{subsec:right-tail-lower}

If \(\lambda_{1,n}\le c/n\), then \((n\lambda_{1,n})^{-1}\ge c^{-1}\).
Choosing \(c\) sufficiently small therefore forces
\(H_n>h_{\kappa,\rho}+\delta\). It remains to estimate the probability of this
microscopic eigenvalue event.

\begin{proposition}
\label{prop:right-tail-lower}
Assume the nondegeneracy-near-zero condition B.3 in
Assumption~\ref{ass:general-entries}.
For every fixed \(\delta>0\),
\[
    \liminf_{n\to\infty}
    \frac{1}{n\log n}
    \log\mathbb P(H_n>h_{\kappa,\rho}+\delta)
    \ge
    -\frac{\rho}{2}.
\]
\end{proposition}

\begin{proof}[Proof of Proposition~\ref{prop:right-tail-lower}.]
The lower bound constructs an explicit event that is already sufficient for the
right-tail deviation. Namely, the event \(\lambda_{1,n}\le c_\delta/n\) already forces the required upward deviation. Choose \(c_\delta>0\) sufficiently small so that
\(c_\delta^{-1}>h_{\kappa,\rho}+\delta\). On the event
\(\lambda_{1,n}\le c_\delta/n\), the first eigenvalue alone gives
\[
    H_n
    \ge
    \frac{1}{n\lambda_{1,n}}
    \ge
    c_\delta^{-1}
    >
    h_{\kappa,\rho}+\delta .
\]
Therefore $\{\lambda_{1,n}\le c_\delta/n\}\subseteq \{H_n>h_{\kappa,\rho}+\delta\}$ and
\[
    \mathbb P(H_n>h_{\kappa,\rho}+\delta)
    \ge
    \mathbb P(\lambda_{1,n}\le c_\delta/n).
\]

It remains to prove that the probability of $\{\lambda_{1,n}\le c_\delta/n\}$ is no smaller than $\exp((-\rho/2+o(1))n\log n)$, and this is characterized by the following Lemma~\ref{lem:lambda1-microscopic-lower}.

\begin{lemma}[Microscopic lower bound for the smallest eigenvalue]
\label{lem:lambda1-microscopic-lower}
\label{lem:microscopic-smallest-eigenvalue-lower}
Assume the nondegeneracy-near-zero condition B.3 in
Assumption~\ref{ass:general-entries}.
For every fixed \(c>0\),
\[
    \liminf_{n\to\infty}
    \frac{1}{n\log n}
    \log\mathbb P\left(\lambda_{1,n}\le \frac{c}{n}\right)
    \ge
    -\frac{\rho}{2}.
\]
\end{lemma}
\begin{proof}[Proof of Lemma~\ref{lem:lambda1-microscopic-lower}.]
Write \(G_n=[g_1,\ldots,g_{q_n}]\in\mathbb R^{N_n\times q_n}\). Let
\(\mathcal H_n:=\operatorname{span}\{g_1,\ldots,g_{q_n-1}\}\) and
\(D_n:=\operatorname{dist}(g_{q_n},\mathcal H_n)\). The point of this
geometric reduction is that one column falling close to the span of the others
creates one very small singular value. By the min--max principle,
\[
    \lambda_{1,n}
    =
    \frac{s_{\min}(G_n)^2}{n}
    \le
    \frac{D_n^2}{n}.
\]
Set \(a_c:=\min\{\sqrt c,u_0/2\}\). Then
\(\{D_n\le a_c\}\subseteq\{\lambda_{1,n}\le c/n\}\).

We now estimate this distance event conditionally on the first \(q_n-1\)
columns. Since the entries have densities, those columns are linearly
independent almost surely, and hence
\(d_n:=\operatorname{codim}(\mathcal H_n)=N_n-(q_n-1)=r_n+1\).
Moreover, \(g_{q_n}\) is independent of \(\mathcal H_n\). Applying
Proposition~\ref{cor:distance-subspace-lower-small-ball} conditionally with
\(H=\mathcal H_n\), \(d=d_n\), and \(a=a_c\), we get, almost surely,
\[
    \mathbb P(D_n\le a_c\mid \mathcal H_n)
    \ge
    k_{\rm low}^{N_n}u_0^{N_n-d_n}
    \operatorname{Vol}_{d_n}(B_2^{d_n}(a_c)).
\]
Therefore \(\mathbb P\left(\lambda_{1,n}\le \frac{c}{n}\right)
    \ge
    k_{\rm low}^{N_n}u_0^{N_n-d_n}
    \operatorname{Vol}_{d_n}(B_2^{d_n}(a_c)).\)
Using
\(\operatorname{Vol}_{d}(B_2^d(a_c))=\pi^{d/2}a_c^d/\Gamma(d/2+1)\) and
Stirling's formula $\log \operatorname{Vol}_{d_n}(B_2^{d_n}(a_c))
    =
    -\frac{d_n}{2}\log d_n + O_c(d_n),$ the remaining prefactor contributes only \(O(N_n)=O(n)\). Since
\(d_n/n=(r_n+1)/n\to \rho\) and \(\log d_n=\log n+O(1)\),
we obtain
\[
    \log\mathbb P\left(\lambda_{1,n}\le \frac{c}{n}\right)
    \ge
    -\frac{\rho}{2}n\log n+O_c(n).
\]
Dividing by \(n\log n\) proves the claim.
\end{proof}

\end{proof}

\subsubsection{Upper bound: excluding non-leading mechanisms}
\label{subsec:right-tail-upper}
\label{subsec:non-hard-edge}
\label{subsec:microscopic-hard-edge-counting}

The goal of this subsection is to show that the probability of the inverse trace exceeding a fixed deviation is bounded above by $\exp((-\rho/2+o(1))n\log n)$, rigorously stated in Proposition~\ref{prop:right-tail-upper}.

\begin{proposition}
\label{prop:right-tail-upper}
Assume the normalization, bounded-density, and log-Sobolev conditions B.1,
B.2, and B.4 in
Assumption~\ref{ass:general-entries}.
For every fixed \(\delta>0\),
\[
    \limsup_{n\to\infty}
    \frac1{n\log n}
    \log\mathbb P(H_n>h_{\kappa,\rho}+\delta)
    \le
    -\frac{\rho}{2}.
\]
\end{proposition}

The proof separates the nonmicroscopic and microscopic contributions. The
nonmicroscopic estimate and the bounded-Lipschitz concentration input are
proved in the main text; the hard-edge pigeonhole and counting inputs
(Lemma~\ref{lem:weighted-hard-edge-pigeonhole} and
Proposition~\ref{prop:microscopic-hard-edge-counting}) are stated where they
are used, with their proofs deferred to
Section~\ref{sec:proofs-auxiliary-estimates}. 

\begin{proof}[Proof of Proposition~\ref{prop:right-tail-upper}.]


\begin{center}
\resizebox{\textwidth}{!}{%
\begin{tikzpicture}[x=1cm,y=1cm,font=\scriptsize]
    \draw[->,thick] (0,0) -- (16.0,0) node[right] {$u$};
    \draw[line width=3pt,purple!65] (0,0) -- (5.50,0);
    \draw[line width=3pt,blue!65!black] (5.50,0) -- (15.25,0);

    \draw[thick] (0,0.11) -- (0,-0.11) node[below] {$0$};
    \node[align=left,below] at (0,0.72) {hard edge};
    \draw[thick,dashed] (0, 0) -- (0,0.30);
    \draw[thick] (5.50,0.11) -- (5.50,-0.11) node[below] {$\eta_n=n^{-\chi}$};
    \draw[thick,dashed] (10.15,0.30) -- (10.15,-0.20);
    \node[align=center,below] at (10.15,0.72) {MP lower edge};
    \node[below] at (10.15,-0.02) {$\ell_{\kappa,\rho}$};
    \node[below] at (15.25,-0.02) {$\infty$};

    \node[
        draw=purple!65,
        fill=purple!6,
        rounded corners=1pt,
        align=center,
        text width=3.70cm,
        anchor=south
    ] at (2.60,0.88)
    {\textbf{Microscopic hard edge}\\
 \textbf{Aim}:\(\mathbb P(S_n>\delta/2)\le
    \exp\{-(\rho/2-o(1))n\log n\}\)\\
    \textbf{Technique}: weighted pigeonhole\\
    \(+\) microscopic hard-edge\\
    counting bound};

    \node[
        draw=blue!65!black,
        fill=blue!6,
        rounded corners=1pt,
        align=center,
        text width=5.85cm,
        anchor=south
    ] at (10.20,0.88)
    {\textbf{Nonmicroscopic spectrum}\\
  \textbf{Aim}: \(\mathbb P(M_n>h_{\kappa,\rho}+\delta/2)\le
    \exp\{-c n^2\eta_n^4\}\)\\
    =\(\exp\{-\omega(n\log n)\}\)\\
    \textbf{Technique}: concentration of bounded Lipschitz spectral statistic};
\end{tikzpicture}%
}
\end{center}

We first decompose the inverse trace into two spectral regions and set
\(\eta_n:=n^{-\chi}\), with fixed \(0<\chi<1/4\). The cutoff satisfies
\(\eta_n\gg n^{-1}\) and separates the microscopic hard-edge contribution from
the nonmicroscopic part. More rigorously, \(M_n\) contains all eigenvalues at
least \(\eta_n\), while \(S_n\) contains the microscopic eigenvalues below
\(\eta_n\). Define
\[
\begin{aligned}
    M_n:=
    \frac1n\sum_{\lambda_{i,n}\ge\eta_n}
    \frac1{\lambda_{i,n}},
    \qquad
    S_n:=\frac1n\sum_{\lambda_{i,n}<\eta_n}
    \frac1{\lambda_{i,n}}.
\end{aligned}
\]
Then \(H_n=M_n+S_n\). If
\(M_n\le h_{\kappa,\rho}+\delta/2\) and \(S_n\le\delta/2\), then
\(H_n\le h_{\kappa,\rho}+\delta\).
Therefore
\[
\begin{aligned}
    \left\{
        H_n>h_{\kappa,\rho}+\delta
    \right\}
    \subseteq
    \left\{M_n>h_{\kappa,\rho}+\frac{\delta}{2}\right\}
    \cup
    \left\{S_n>\frac{\delta}{2}\right\}.
\end{aligned}
\]
This reduces the right-tail event to a nonmicroscopic deviation and a
microscopic hard-edge deviation.

\paragraph{Nonmicroscopic control for \(M_n\).}

Eigenvalues above the microscopic cutoff are negligible at the \(n\log n\)
scale.
We first state the
concentration estimate used to control such bounded-Lipschitz spectral
statistics.

\begin{lemma}[Concentration of Lipschitz spectral statistics]
\label{lem:lipschitz-spectral-statistic-concentration}
Let \(G_n\in\mathbb R^{N_n\times q_n}\) be the random matrix in
\eqref{eq:define-rn}, and assume that its entries satisfy the log-Sobolev
condition B.4 in
Assumption~\ref{ass:general-entries}. For a bounded Lipschitz function
\(h:\mathbb R_+\to\mathbb R\) and a deterministic matrix
\(B\in\mathbb R^{N_n\times q_n}\), define
\[
    F_{n,h}(B):=\frac1n\sum_{i=1}^{q_n}h(\lambda_i(n^{-1}B^\top B)).
\]
There exist constants \(c,C>0\) and \(n_0<\infty\) such that, for every
\(n\ge n_0\), every bounded Lipschitz \(h\), and every
$
    t\ge C\|h\|_\infty e^{-c n^2},
$
\begin{equation}
    \mathbb P\bigl(F_{n,h}(G_n)-\mathbb E F_{n,h}(G_n)>t\bigr)
    \le
    e^{-c n^2}
    +
    \exp\left\{
        -c\frac{n^2t^2}{\|h\|_{\mathrm{Lip}}^2}
    \right\}.
    \label{eq:lipschitz-concentration-spectral-statistic}
\end{equation}
If \(\|h\|_{\mathrm{Lip}}=0\), the second term is interpreted as \(0\).
\end{lemma}

\begin{proof}[Proof of Lemma~\ref{lem:lipschitz-spectral-statistic-concentration}.]
If \(\|h\|_{\mathrm{Lip}}=0\), then \(F_{n,h}(G_n)\) is deterministic and the
claim is immediate. Assume below that \(\|h\|_{\mathrm{Lip}}>0\).
Choose \(R>0\) large enough so that the Frobenius ball
\(B_R:=\{G:\|G\|_F\le Rn\}\) satisfies
\(\mathbb P(G_n\notin B_R)\le e^{-c n^2}\). This follows from the fact that
the log-Sobolev condition B.4 implies subgaussian tails, together with
\(N_nq_n=O(n^2)\). Let \(\Pi_R\) be the Euclidean
projection onto \(B_R\), and set
\(F^{(R)}_{n,h}(G):=F_{n,h}(\Pi_R G)\). By
Lemma~\ref{lem:local-lipschitz-spectral-statistic}, applied with
\(\omega_n=1/n\), and since \(\Pi_R\) is nonexpansive, the map
\(F^{(R)}_{n,h}\) is globally
\(C\|h\|_{\mathrm{Lip}}/n\)-Lipschitz. The logarithmic Sobolev inequality in
Assumption B.4, together with tensorization, therefore gives~\citep{ledoux2006concentration}
\[
    \mathbb P\left(
        F^{(R)}_{n,h}(G_n)-\mathbb E F^{(R)}_{n,h}(G_n)>u
    \right)
    \le
    \exp\left\{
        -c\frac{n^2u^2}{\|h\|_{\mathrm{Lip}}^2}
    \right\}.
\]
Since \(q_n/n=O(1)\), both \(|F_{n,h}(G_n)|\) and
\(|F^{(R)}_{n,h}(G_n)|\) are bounded by \(C_0\|h\|_\infty\), and
\(F^{(R)}_{n,h}(G_n)=F_{n,h}(G_n)\) on the event \(\{G_n\in B_R\}\). Hence
\[
    |\mathbb E F^{(R)}_{n,h}(G_n)-\mathbb E F_{n,h}(G_n)|
    \le
    2C_0\|h\|_\infty\mathbb P(G_n\notin B_R)
    \le
    C\|h\|_\infty e^{-c n^2}.
\]
If \(t\ge C\|h\|_\infty e^{-c n^2}\), then this expectation error is at most
\(t/2\), after increasing \(C\) if necessary, and
\eqref{eq:lipschitz-concentration-spectral-statistic} follows by applying the
preceding concentration inequality with \(u=t/2\) and by adding
\(\mathbb P(G_n\notin B_R)\).
\end{proof}

\begin{proposition}[Nonmicroscopic inverse-trace contribution]
\label{prop:nonmicroscopic-inverse-trace}
Assume the normalization and log-Sobolev conditions B.1 and B.4 in
Assumption~\ref{ass:general-entries}. Let
\(\eta_n=n^{-\chi}\) with \(0<\chi<1/4\), and set \(M_n:=\frac1n\sum_{\lambda_{i,n}\ge\eta_n}
        \frac1{\lambda_{i,n}}\). For every fixed \(\delta>0\), there
exists \(c_\delta>0\) such that
\[
    \mathbb P\left(
        M_n
        >
        h_{\kappa,\rho}+\delta
    \right)
    \le
    \exp\{-c_\delta n^2\eta_n^4\}
\]
for all sufficiently large \(n\). In particular, this probability is
\(\exp\{-\omega(n\log n)\}\).
\end{proposition}

\begin{proof}[Proof of Proposition~\ref{prop:nonmicroscopic-inverse-trace}.]
Set
\[
    g_n(u):=
    \begin{cases}
        \eta_n^{-1}, & 0\le u<\eta_n,\\
        u^{-1}, & u\ge \eta_n.
    \end{cases}
\]

Since \(g_n(u)\ge u^{-1}\mathbf 1\{u\ge\eta_n\}\), we have
\(M_n\le F_{n,g_n}(G_n)\). Moreover, \(\lVert g_n\rVert_\infty\le C\eta_n^{-1}\) and \(\lVert g_n\rVert_{\mathrm{Lip}}\le C\eta_n^{-2}\).

Since \(M_n\le F_{n,g_n}(G_n)\) by our construction, \(\{M_n>h_{\kappa,\rho}+\delta\}\subseteq \{F_{n,g_n}(G_n)>h_{\kappa,\rho}+\delta\}\), therefore
\begin{equation}\label{inequality:gnconcerntration}
\mathbb P\left(M_n>h_{\kappa,\rho}+\delta\right)
    \le \mathbb P(F_{n,g_n}(G_n)>h_{\kappa,\rho}+\delta).
\end{equation}
We aim to use the concentration of \(F_{n,g_n}(G_n)\) to bound the right-hand side.
First, we estimate the expectation \(\mathbb E F_{n,g_n}(G_n)\).  We fix
\(c_0>0\) as stated in
Lemma~\ref{lem:rectangular-lower-edge-tail}. Let
\(f_0\in C_b^1(\mathbb R_+)\) be a fixed nonnegative function such that
\(f_0(u)=u^{-1}\) for all \(u\ge c_0\). By construction,
\[
    F_{n,g_n}(G_n)
    \le
    \frac1n\sum_{i=1}^{q_n}f_0(\lambda_{i,n})
    +
    \frac{q_n}{n}\eta_n^{-1}
    \mathbf 1_{\{\lambda_{1,n}<c_0\}}.
\]
The first term converges in expectation to \(h_{\kappa,\rho}\) by the
universality of the Marchenko--Pastur law under condition B.1~\citep{bai2010spectral},
the boundedness of \(f_0\in C_b^1(\mathbb R^+)\), \(q_n/n\to\kappa\), and the
identity \(f_0(u)=u^{-1}\) on the limiting support.
The log-Sobolev condition B.4 implies subgaussian tails, so
Lemma~\ref{lem:rectangular-lower-edge-tail} applies, which implies the following control of the second term:
\[
    \frac{q_n}{n}\eta_n^{-1}
    \mathbb P(\lambda_{1,n}<c_0)
    \le
    C\eta_n^{-1}e^{-cn}
    =
    o(1).
\]
Thus \(\mathbb E F_{n,g_n}(G_n)\le h_{\kappa,\rho}+\delta/2\) for all
sufficiently large \(n\). For each \(n\), apply
Lemma~\ref{lem:lipschitz-spectral-statistic-concentration} with \(h=g_n\).
Since \(\|g_n\|_\infty=O(\eta_n^{-1})\) is polynomial in \(n\), the lower
threshold condition
\[
    \delta/2\ge C\|g_n\|_\infty e^{-cn^2}
\]
holds for all large \(n\). Hence combining~\eqref{inequality:gnconcerntration}
with Lemma~\ref{lem:lipschitz-spectral-statistic-concentration} we have
\begin{equation}
\begin{aligned}[b]
    \mathbb P\left(M_n>h_{\kappa,\rho}+\delta\right)
    \le
    \mathbb P\left(
        F_{n,g_n}(G_n)-\mathbb E F_{n,g_n}(G_n)>\frac{\delta}{2}
    \right)
    \le
    e^{-cn^2}
    +
    \exp\{-c_\delta n^2\eta_n^4\}.
\end{aligned}
\end{equation}
Since \(\eta_n\le1\), the \(e^{-cn^2}\) term and the resulting factor \(2\) are
absorbed into the second term after decreasing \(c_\delta\).
\end{proof}

Applying Proposition~\ref{prop:nonmicroscopic-inverse-trace} with
\(\delta/2\) in place of \(\delta\) gives
\[
    \mathbb P\left(M_n>h_{\kappa,\rho}+\frac{\delta}{2}\right)
    \le
    \exp\{-c_\delta n^2\eta_n^4\}
    =
    \exp\{-\omega(n\log n)\}.
\]

\paragraph{Microscopic hard-edge control for \(S_n\).}

This is the only region that can still have the same \(n\log n\) scale as the
lower bound. The event is \(\{S_n>\delta/2\}\), where \(S_n\) is the
contribution of eigenvalues below \(\eta_n\). One eigenvalue at scale \(1/n\)
already gives an \(O(1)\) contribution to \(H_n\), while forcing \(k\)
eigenvalues to be that small should cost more at the \(n\log n\) scale. The
proof first uses Lemma~\ref{lem:weighted-hard-edge-pigeonhole} to convert a
large value of \(S_n\) into a union of ordered-eigenvalue events, and then
applies Proposition~\ref{prop:microscopic-hard-edge-counting} to show that
\(k=1\) is the leading term.

\begin{proposition}[Microscopic hard-edge counting]
\label{prop:microscopic-hard-edge-counting}
Assume the normalization and bounded-density conditions B.1 and B.2 in
Assumption~\ref{ass:general-entries}, and assume in addition that the entries
have subgaussian tails. Let \(L_n\to\infty\) be a deterministic real
sequence and \(r_n\) be as defined in~\eqref{eq:define-rn}. Then there exist
\(c,C\in(0,\infty)\) and \(n_0<\infty\) such that, for every
\(n\ge n_0\), every \(1\le k\le q_n\), and every \(0<t\le1\),
\[
    \mathbb P(\lambda_{k,n}\le t)
    \le
    \exp\{Cnk(1+\log L_n)\}\,
    t^{k(r_n+k)/2}
    +
    2\exp\{-c nL_n^2\}
\].

In particular, for every \(M>0\), there exist constants \(C_M<\infty\) and
\(n_0<\infty\) such that, for every \(n\ge n_0\), every
\(1\le k\le q_n\), and every \(0<t\le1\),
\[
    \mathbb P(\lambda_{k,n}\le t)
    \le
    \exp\{C_M nk\log\log n\}\,
    t^{k(r_n+k)/2}
    +
    \exp\{-M n\log n\}
\].
\end{proposition}


Then we introduce the following weighted hard-edge pigeonhole lemma. This pigeonhole construction serves two purposes. First, we assign moderately increasing thresholds to eigenvalues of increasing order, so that the probabilities of the ordered eigenvalues falling below their assigned thresholds are approximately balanced, while the threshold for the first eigenvalue is still of order \(O(n^{-1})\) as in the lower bound. Thus the pigeonhole events take the form of \(\lambda_k\lesssim \frac{ k^\theta}{n}\). Indeed, if \(S_\eta=\frac{1}{n}\sum_{i=1}^{K_\eta}\lambda_i^{-1}>u,\)
then there must exist some index \(1\le k\le K_\eta:=\#\{i:\lambda_i<\eta\}\) such that \(\lambda_k \lesssim \frac{k^\theta}{n}\). Second, we replace the random hard-edge eigenvalue count \(K_\eta\) by a deterministic cutoff, thus we impose the explicit constraint \(\lambda_k<\eta\) as part of the pigeonhole event. With this design, we have


\begin{lemma}[Weighted hard-edge pigeonhole]
\label{lem:weighted-hard-edge-pigeonhole}
Let \(0<\lambda_1\le\cdots\le\lambda_q\). For \(u,\eta>0\) and
\(\theta>1\), set
\(K_\eta:=\#\{i:\lambda_i<\eta\}\),
\(S_\eta:=n^{-1}\sum_{i=1}^{K_\eta}\lambda_i^{-1}\), and
\(M_\theta:=\sum_{j=1}^{\infty}j^{-\theta}<\infty\).
Then
\[
    \{S_\eta>u\}
    \subseteq
    \bigcup_{k=1}^q
    \left\{
        \lambda_k
        <
        \eta\wedge \frac{M_\theta k^\theta}{nu}
    \right\}.
\]
\end{lemma}

The proof is deferred to Section~\ref{sec:proofs-auxiliary-estimates}. 

After applying Lemma \ref{lem:weighted-hard-edge-pigeonhole} with \(\eta=\eta_n\), \(u=\delta/2\), some fixed $\theta>1$, and
\(A=M_\theta/u\), we introduce the deterministic cutoff
\(K_n:=
    \left\lfloor
        \left(\frac{\eta_n n}{A}\right)^{1/\theta}
    \right\rfloor \vee 1 .\) This cutoff is used only to compare the two thresholds
\(A k^\theta/n\) and \(\eta_n\). For \(k\le K_n\), we have
\(A k^\theta/n\le \eta_n\), so the corresponding event is controlled by
\(\{\lambda_{k,n}<A k^\theta/n\}\). For \(k>K_n\), we have
\(A k^\theta/n>\eta_n\), so the event reduces to
\(\{\lambda_{k,n}<\eta_n\}\). By the monotonicity of the ordered
eigenvalues, the union over all \(k>K_n\) is contained in the single event \(\{\lambda_{K_n,n}<\eta_n\}.\) Therefore,
\[
    \mathbb P(S_n>u)
    \le
    \sum_{k=1}^{K_n}
    \mathbb P\left(
        \lambda_{k,n}\le \frac{A k^\theta}{n}
    \right)
    +
    \mathbb P\left(
        \lambda_{K_n,n}\le \eta_n
    \right).
\]

The log-Sobolev condition B.4 implies subgaussian tails, so the
subgaussian-tail requirement in
Proposition~\ref{prop:microscopic-hard-edge-counting} is satisfied.
Since \(\eta_n=n^{-\chi}\), we have
\(K_n=n^{(1-\chi)/\theta+o(1)}\); hence \(K_n\to\infty\), \(K_n=o(n)\), and,
because \(q_n\asymp n\), \(K_n\le q_n\) for all sufficiently large \(n\).
We now apply the hard-edge counting bound Proposition~\ref{prop:microscopic-hard-edge-counting} with a fixed \(M>\rho/2+1\). After
increasing the lower bound on \(n\) once, the estimate applies throughout the
following union over \(1\le k\le K_n\) and again to the \(k=K_n\) term below.
For \(1\le k\le K_n\), let
\begin{equation}\label{eq:definition-of-tnk-Lnk}
    t_{n,k}:=\frac{A k^\theta}{n},
    \qquad
    L_{n,k}:=\log\frac1{t_{n,k}}
    =
    \log n-\log A-\theta\log k.
\end{equation}
Since \(t_{n,k}\le\eta_n=n^{-\chi}<1\),
we apply Proposition~\ref{prop:microscopic-hard-edge-counting} with \(t=t_{n,k}\), which gives, for a constant
\(C<\infty\),
\[
\begin{aligned}
    \mathbb P(\lambda_{k,n}\le t_{n,k})
    \le
    P_{n,k}
    +
    e^{-M n\log n},\quad
    P_{n,k}
    :=
    \exp\left\{
        Cnk\log\log n
        -
        \frac{k(r_n+k)}2 L_{n,k}
    \right\}.
\end{aligned}
\]
The summation over \(k\) is where the single-eigenvalue mechanism becomes
visible. Now we separate the summation from \(1\) to \(K_n\) into three regimes:
\begin{itemize}
    \item  \emph{Regime I: \(k=1\).}
Here \(L_{n,1}=\log n-\log A\), so
\[
\begin{aligned}
    \log P_{n,1}
    &=
    Cn\log\log n
    -
    \frac{r_n+1}{2}(\log n-\log A)
    =
    -\left(\frac{\rho}{2}-o(1)\right)n\log n.
\end{aligned}
\]
Thus \(P_{n,1}\le
\exp\{-(\rho/2-o(1))n\log n\}\). Our next goal is to show that this is the leading term, in the sense that for all $2\le k\le K_n$, $\log P_{n,k}\le \log P_{n,1}$ for sufficiently large $n$.

\item \emph{Regime II: \(2\le k\le m_n\).} Here we choose $m_n:=\lfloor n^\beta\rfloor$ with  $0<\beta<\min\{(2\theta)^{-1},(1-\chi)/\theta\}$. Then \(m_n<K_n\) for all sufficiently large \(n\). Uniformly over
\(2\le k\le m_n\), we have
\[
    L_{n,k}\ge (1-\theta\beta-o(1))\log n,
    \qquad
    Cnk\log\log n=o(kn\log n).
\]
Since \(r_n/n\to\rho\), this yields
\[
    \log P_{n,k}
    \le
    -
    \left(
        \frac{\rho}{2}(1-\theta\beta)-o(1)
    \right)kn\log n.
\]
Because \(k\ge2\) and \(\beta<1/(2\theta)\), there exists
\(\zeta_\beta>0\) such that, for all sufficiently large \(n\),
\[
    \sum_{2\le k\le m_n}P_{n,k}
    \le
    \exp\left\{
        -\left(\frac{\rho}{2}+\zeta_\beta\right)n\log n
    \right\}.
\]
Here we used the fact that \(\log m_n=o(n\log n)\).

\item \emph{Regime III: \(m_n<k\le K_n\).}
Recall the definition of $t_{n,k}$ and $L_{n,k}$ in~\eqref{eq:definition-of-tnk-Lnk}.
For every \(k\le K_n\), \(t_{n,k}\le\eta_n=n^{-\chi}\), hence
\(L_{n,k}\ge\chi\log n\). Also \(r_n\ge(\rho/2)n\) for all sufficiently large
\(n\). Therefore, uniformly over \(m_n<k\le K_n\),
\[
    \log P_{n,k}
    \le
    Cnk\log\log n
    -
    \frac{\rho\chi}{4}nk\log n
    \le
    -c nk\log n
\]
for some \(c>0\). Since \(m_n=n^{\beta+o(1)}\) and \(K_n\le n\) for large
\(n\),
\[
    \sum_{m_n<k\le K_n}P_{n,k}
    \le
    K_n e^{-c n m_n\log n}
    =
    \exp\{-\omega(n\log n)\}.
\]
\end{itemize}
Combining the three regimes above, together with the
additive errors from the counting bound,
\[
    \sum_{k=1}^{K_n}
    \mathbb P\left(
        \lambda_{k,n}\le \frac{A k^\theta}{n}
    \right)
    \le
    \exp\left\{
        -\left(\frac{\rho}{2}-o(1)\right)n\log n
    \right\},
\]
because
\(K_n e^{-M n\log n}=\exp\{-(M-o(1))n\log n\}\) and
\(M>\rho/2+1\).

It remains to bound the large-\(k\) term. Applying the same counting bound Proposition~\ref{prop:microscopic-hard-edge-counting} with
\(k=K_n\) and \(t=\eta_n=n^{-\chi}\) gives
\[
\begin{aligned}
    \mathbb P(\lambda_{K_n,n}\le \eta_n)
    &\le
    \exp\left\{
        CnK_n\log\log n
        -
        \frac{K_n(r_n+K_n)}2
        \chi\log n
    \right\}
    +
    e^{-M n\log n}.
\end{aligned}
\]
The first term is \(\exp\{-\omega(n\log n)\}\), because
\(r_n\ge(\rho/2)n\) and \(K_n\to\infty\); the second term is negligible by the
choice of \(M\). Consequently,
\[
    \mathbb P(S_n>u)
    \le
    \exp\left\{
        -\left(\frac{\rho}{2}-o(1)\right)n\log n
    \right\},
\]
or equivalently \(\limsup_{n\to\infty}
    \frac1{n\log n}
    \log
    \mathbb P\left(S_n>\frac{\delta}{2}\right)
    \le
    -\frac{\rho}{2}. \)

Combining the estimates,
and using the union bound we get

\[
\begin{aligned}
    \mathbb P(H_n>h_{\kappa,\rho}+\delta)
    \le
    \exp\{-\omega(n\log n)\}
    +
    \exp\left\{
        -\left(\frac{\rho}{2}-o(1)\right)n\log n
    \right\}.
\end{aligned}
\]
Therefore $\limsup_{n\to\infty}
    \frac1{n\log n}
    \log\mathbb P(H_n>h_{\kappa,\rho}+\delta)
    \le
    -\frac{\rho}{2}.$ This proves the proposition.
\end{proof}

Together with Proposition~\ref{prop:right-tail-lower}, this proves
\eqref{eq:proportional-right-tail-asymptotic}.

\begin{remark}[Fixed-dimensional and proportional right-tail proof strategies]
\label{rem:fixed-distance-proof-scope}
\label{rem:cheapest-right-tail-mechanism}
The fixed-dimensional and proportional right-tail proofs identify the same
qualitative hard-edge mechanism, but require different proof strategies. In the
fixed-dimensional proof in Section~\ref{sec:proof-thm1}, the negative-second-moment identity converts the
inverse trace into a sum of inverse squared column-to-subspace distances.
Because \(q\) is fixed, a union bound and the matching small-ball estimates for
one distance are sharp enough: one column of \(G\) becoming nearly linearly
dependent on the remaining columns gives the correct polynomial exponent as
\(x\to\infty\).

In the proportional regime, the same distance-to-subspace pigeonhole argument
is too coarse. Since \(q=q_n\asymp n\), applying the same identity to the
analogous distances \(D_i\) gives only
    $\{H_n>h_{\kappa,\rho}+\delta\}
    \subseteq
    \bigcup_{i\le q_n}
    \left\{
        D_i^2<\frac{q_n}{h_{\kappa,\rho}+\delta}
    \right\}$.
The threshold on the right is a constant multiple of \(n\), below the typical
value of \(D_i^2\) on its \(n\)-scale. Thus this is a macroscopic lower-tail
event for a column-distance variable, not the microscopic hard-edge event that
produces the sharp \(n\log n\) cost. Therefore, this naive pigeonhole gives
only a nonmatching exponential-in-\(n\) upper bound. The proof in
Section~\ref{sec:proof-thm2}
therefore works directly with ordered eigenvalues. The lower bound constructs
the event \(\lambda_{1,n}=O(n^{-1})\), which already creates an \(O(1)\) upward
deviation. The upper bound separates the nonmicroscopic contribution from the
microscopic hard-edge contribution and rules out cheaper alternatives: forcing
\(k\) eigenvalues to the hard edge costs at least order
\(\exp\{-c k n\log n\}\), up to lower-order factors. Hence the cheapest
proportional right-tail mechanism is \(k=1\), namely a single eigenvalue
collapsing to the hard edge at scale \(n^{-1}\).
\end{remark}

\section{High-dimensional left-tail large deviations under Gaussian entries}
\label{sec:gaussian-left-tail}

Under Gaussian entries, the two tails arise from different spectral mechanisms. The right tail is driven by a microscopic hard-edge eigenvalue, whereas the left tail requires a collective spectral deformation and therefore has speed \(n^2\). Thus the two tails are not mirror images of each other, and their large-deviation speeds need not coincide. Throughout this section, we work in the same proportional regime as in Section~\ref{sec:proportional-right-tail}.

In particular,
\[
    \frac{q_n}{n}\to\kappa\in(0,\infty),
    \qquad
    \frac{r_n}{n}\to\rho\in(0,\infty).
\]
In this regime,
the empirical spectral measure
\(\widehat\mu_n=q_n^{-1}\sum_{i=1}^{q_n}\delta_{\lambda_{i,n}}\) satisfies the
Gaussian Wishart spectral LDP with rate function \(I_{\kappa,\rho}\) defined in~\eqref{eq:wishart-rate-prelim}, and for a measure \(\mu\in\mathcal P(\mathbb R_+)\), we write the inverse-moment functional as
\(T(\mu):=\int u^{-1}\,d\mu(u)\in[0,+\infty]\).

\subsection{Large deviation left-tail of ridgeless regression risk}
\begin{theorem}[Gaussian left-tail large deviations for ridgeless regression]
\label{thm:ridgeless-left-tail}
Assume Assumption~\ref{ass:gaussian-entries}. In the ridgeless regression
specialization defined in Section~\ref{sec:random-matrix-4-ridgeless}, let
\(\kappa_\gamma, \rho_\gamma\) be as in~\eqref{eq:define-kappa_gamma} and
\(b_\gamma\) as in~\eqref{eq:typical-risk-prelim}. Then, for every
\(x>b_\gamma\),
\[
    \lim_{n\to\infty}
    \frac1{n^2}\log\mathbb P(R_n\le x)
    =
    \lim_{n\to\infty}
    \frac1{n^2}\log\mathbb P(R_n<x)
    =
    -\mathcal J_{\kappa_\gamma,\rho_\gamma}
    \left(\frac{x-b_\gamma}{\kappa_\gamma}\right),
\]
where $\mathcal J_{\kappa,\rho}$ is defined as $\mathcal J_{\kappa,\rho}(s)
    :=
    \inf\{I_{\kappa,\rho}(\mu):\mu\in\mathcal P(\mathbb R_+),\ T(\mu)\le s\}$ for $s>0$.

\end{theorem}

Although \(T(\mu)=\int u^{-1}\,d\mu\) is not weakly continuous, the obstruction comes from mass approaching zero and is therefore relevant to upward deviations rather than the left-tail event studied here. The upper bound makes this rigorous by replacing \(T\) with the capped functionals \(T_M(\mu)=\int \min\{u^{-1},M\}\,d\mu(u)\), which are bounded and continuous. The lower bound approximates near-minimizing measures by compactly supported piecewise-uniform measures, constructs local spectral-box events around them, and evaluates their probabilities using the Wishart eigenvalue density.

\subsection{Proof of Theorem~\ref{thm:ridgeless-left-tail}}\label{sec:proof-thm-lefttail}

Fix a risk threshold
\(a>b_\gamma\), and let \(y:=a-b_\gamma>0\). In the regression specialization,
\(R_n=b_n+H_n\) and \(b_n\to b_\gamma\). Fix \(\varepsilon\in(0,y)\). For all
sufficiently large \(n\),
\[
    \{H_n<y-\varepsilon\}
    \subseteq
    \{R_n<a\}
    \subseteq
    \{R_n\le a\}
    \subseteq
    \{H_n\le y+\varepsilon\}.
\]
Therefore it suffices to prove the following high-dimensional large deviation of the left tail of the inverse trace statistics:
for every \(x>0\),
\begin{equation}
\label{eq:gaussian-left-tail-asymptotic}
    \lim_{n\to\infty}
    \frac1{n^2}\log\mathbb P(H_n\le x)
    =
    \lim_{n\to\infty}
    \frac1{n^2}\log\mathbb P(H_n<x)
    =
    -\mathcal J_{\kappa,\rho}(x/\kappa).
\end{equation}
Indeed, using \eqref{eq:gaussian-left-tail-asymptotic} with
\((\kappa,\rho)=(\kappa_\gamma,\rho_\gamma)\), the inclusions above give
\[
\begin{aligned}
    \liminf_{n\to\infty}\frac1{n^2}\log\mathbb P(R_n<a)
    \ge
    -\mathcal J_{\kappa_\gamma,\rho_\gamma}
    ((y-\varepsilon)/\kappa_\gamma),\quad
    \limsup_{n\to\infty}\frac1{n^2}\log\mathbb P(R_n\le a)
    \le
    -\mathcal J_{\kappa_\gamma,\rho_\gamma}
    ((y+\varepsilon)/\kappa_\gamma).
\end{aligned}
\]
Sending \(\varepsilon\downarrow0\) and using the continuity of
\(s\mapsto\mathcal J_{\kappa_\gamma,\rho_\gamma}(s)\) from
Lemma~\ref{lem:left-variational-rate-continuity}, together with
\(\{R_n<a\}\subseteq\{R_n\le a\}\), gives the theorem, since
\(y=a-b_\gamma\).

It remains to prove \eqref{eq:gaussian-left-tail-asymptotic}.

The proof has two estimates. The upper bound applies the Gaussian spectral LDP to the closed sets defined by \(T_M\) and then lets \(M\to\infty\). The upper bound is stated next.

\begin{proposition}[Left-tail upper bound]
\label{prop:left-tail-upper}
For every \(x>0\),
\[
    \limsup_{n\to\infty}
    \frac1{n^2}\log\mathbb P(H_n\le x)
    \le
    -\mathcal J_{\kappa,\rho}(x/\kappa).
\]
\end{proposition}

\begin{proof}[Proof of Proposition~\ref{prop:left-tail-upper}.]
Write \(\kappa_n:=q_n/n\), so that \(\kappa_n\to\kappa\), and
\(H_n=\kappa_nT(\widehat\mu_n)\). Fix \(\varepsilon>0\). Since
\(x/\kappa_n\to x/\kappa\),
for all sufficiently large \(n\),
\[
    \{H_n\le x\}
    =
    \left\{T(\widehat\mu_n)\le\frac{x}{\kappa_n}\right\}
    \subseteq
    \{T(\widehat\mu_n)\le x/\kappa+\varepsilon\}.
\]
For \(M>0\), set \(g_M(u):=\min\{u^{-1},M\}\) for \(u>0\),
\(g_M(0):=M\), and \(T_M(\mu):=\int g_M(u)\,d\mu(u)\). Then
\(g_M\in C_b(\mathbb R_+)\), \(T_M\le T\), and \(T_M(\mu)\uparrow T(\mu)\).
Thus
\(\{T(\widehat\mu_n)\le x/\kappa+\varepsilon\}\subseteq
\{T_M(\widehat\mu_n)\le x/\kappa+\varepsilon\}\).
The map \(\mu\mapsto T_M(\mu)\) is continuous for the weak topology because
\(g_M\in C_b(\mathbb R_+)\). Hence
\(F_{M,\varepsilon}:=\{\mu\in\mathcal P(\mathbb R_+):T_M(\mu)\le
x/\kappa+\varepsilon\}\) is closed. By the Gaussian Wishart spectral LDP from
Section~\ref{sec:prelim},
\[
    \limsup_{n\to\infty}
    \frac1{n^2}\log\mathbb P(H_n\le x)
    \le
    -\inf_{\mu:T_M(\mu)\le x/\kappa+\varepsilon}I_{\kappa,\rho}(\mu).
\]
Now let \(M\to\infty\). Since \(T_M\uparrow T\), the closed sets
\(F_{M,\varepsilon}\) decrease to \(\{T\le x/\kappa+\varepsilon\}\). Because
\(I_{\kappa,\rho}\) is a good rate function, the infima over this decreasing
family converge to the infimum over the intersection. Therefore
\[
    \limsup_{n\to\infty}
    \frac1{n^2}\log\mathbb P(H_n\le x)
    \le
    -\inf_{\mu:T(\mu)\le x/\kappa+\varepsilon}I_{\kappa,\rho}(\mu)
    =
    -\mathcal J_{\kappa,\rho}(x/\kappa+\varepsilon).
\]
Finally send \(\varepsilon\downarrow0\). By
Lemma~\ref{lem:left-variational-rate-continuity},
\(\mathcal J_{\kappa,\rho}(x/\kappa+\varepsilon)\to
\mathcal J_{\kappa,\rho}(x/\kappa)\), which proves the upper bound.
\end{proof}

The matching lower bound is proved by constructing local box events around
prescribed piecewise-uniform spectral measures. The strict-sublevel
representation in Lemma~\ref{lem:left-variational-rate-continuity} lets us
choose a target measure with strict inverse-moment slack and rate arbitrarily
close to the variational infimum. Lemma~\ref{lem:piecewise-uniform-approximation-left}
then reduces the construction to compactly supported piecewise-uniform
measures. For such a target \(\nu\), the local box event has exact
\(n^2\)-scale exponent \(-I_{\kappa,\rho}(\nu)\).

\begin{proposition}[Left-tail lower bound]
\label{prop:left-tail-lower}
For every \(x>0\),
\[
    \liminf_{n\to\infty}
    \frac1{n^2}\log\mathbb P(H_n<x)
    \ge
    -\mathcal J_{\kappa,\rho}(x/\kappa).
\]
\end{proposition}

\begin{proof}[Proof of Proposition~\ref{prop:left-tail-lower}.]




To prove the lower bound, we construct a set of eigenvalue configurations, called the local box event, that
already lies in the event \(\{H_n<x\}\) and has the desired \(n^2\)-scale probability. 

Choose \(\mu\) near the variational infimum and approximate it by a piecewise-uniform \(\nu\). It then remains to construct a local box event \(B_n\) such that \(n^{-2}\log\mathbb P(B_n)\to-I_{\kappa,\rho}(\nu)\), while \(B_n\subset\{H_n<x\}\) for all large \(n\).

Recall that by definition \(\mathcal J_{\kappa,\rho}(x/\kappa)=\inf_{\mu:T(\mu)\le x/\kappa}I_{\kappa,\rho}(\mu)\), and~\eqref{eq:left-strict-sublevel-rate} of Lemma~\ref{lem:left-variational-rate-continuity} gives that \(\mathcal J_{\kappa,\rho}(x/\kappa)=\inf_{\mu:T(\mu)< x/\kappa}I_{\kappa,\rho}(\mu)\) as well. Fix \(\varepsilon>0\). By~\eqref{eq:left-strict-sublevel-rate}, choose \(\mu\in\mathcal P(\mathbb R_+)\) such that
$T(\mu)<x/\kappa$ and $
    I_{\kappa,\rho}(\mu)
    \le
    \mathcal J_{\kappa,\rho}(x/\kappa)+\varepsilon$.

Also by Lemma~\ref{lem:piecewise-uniform-approximation-left}, there exists a
piecewise-uniform probability measure \(\nu\) with $\operatorname{supp}(\nu)\subset[m,M]\subset(0,\infty)$ such that
    $T(\nu)<x/\kappa$ and
    $I_{\kappa,\rho}(\nu)\le I_{\kappa,\rho}(\mu)+\varepsilon$.
We fix such a \(\nu\) and choose \(\eta>0\) so small that \(T(\nu)+2\eta<x/\kappa\).

It is enough to construct \(B_n\) so that \(|T(\widehat\mu_n)-T(\nu)|\le\eta\) on \(B_n\); then \(B_n\subset\{H_n<x\}\) for all large \(n\). Thus we prove the following lemma.

\begin{lemma}[Local box event lower bound for piecewise-uniform measures]
\label{lem:local-box-event-piecewise-left}
Let \(\nu\in\mathcal P(\mathbb R_+)\) be piecewise uniform and suppose
\(\operatorname{supp}(\nu)\subset[m,M]\subset(0,\infty)\). Then, for every
\(\eta>0\), there exist local box events \(B_n=B_n(\nu,\eta)\) such that, on \(B_n\),
\(|T(\widehat\mu_n)-T(\nu)|\le\eta\) for \(n\) large enough, and
\[
    \lim_{n\to\infty}
    \frac1{n^2}\log\mathbb P(B_n)
    =
    -I_{\kappa,\rho}(\nu).
\]
\end{lemma}

\begin{proof}[Proof of Lemma~\ref{lem:local-box-event-piecewise-left}.]
Since $\nu$ is a piecewise-uniform measure, we can write \(d\nu(u)=\sum_{\ell=1}^L c_\ell\mathbf 1_{J_\ell}(u)\,du\), where the
intervals \(J_\ell=[\alpha_\ell,\beta_\ell]\) are disjoint and contained in
\([m,M]\). Let \(w_\ell:=\nu(J_\ell)=c_\ell|J_\ell|\).
We only keep intervals with \(w_\ell>0\).

Choose integers \(q_{\ell,n}\) such that
\(\sum_{\ell=1}^L q_{\ell,n}=q_n\) and \(q_{\ell,n}/q_n\to w_\ell\).
For each \(\ell\), choose points
\(z_{\ell,1,n},\ldots,z_{\ell,q_{\ell,n},n}\in J_\ell\) as the midpoints of
the \(q_{\ell,n}\) equal subintervals of \(J_\ell\). Set
\(\delta_n:=n^{-3}\), and define
\(B_{\ell,j,n}:=[z_{\ell,j,n}-\delta_n,z_{\ell,j,n}+\delta_n]\cap J_\ell\).
For all sufficiently large \(n\), these boxes are contained in \([m/2,2M]\).

Let \(\xi_{1,n},\ldots,\xi_{q_n,n}\) denote eigenvalue coordinates in the
symmetric density \eqref{eq:wishart-density-prelim}. Define the labelled box
event \(B_n\) by assigning the first \(q_{1,n}\) coordinates to the boxes in
\(J_1\), the next \(q_{2,n}\) coordinates to the boxes in \(J_2\), and so on:
\[
    B_n
    :=
    \bigcap_{\ell=1}^L
    \bigcap_{j=1}^{q_{\ell,n}}
    \{\xi_{i(\ell,j),n}\in B_{\ell,j,n}\},
\]
where \(i(\ell,j)\) is any fixed indexing of the labels. On \(B_n\), the empirical measure
\(\widehat\mu_n\) is supported in \([m/2,2M]\), and the labelled points form a
Riemann approximation of \(\nu\). Since \(u\mapsto u^{-1}\) is uniformly
continuous on \([m/2,2M]\),
\(T(\widehat\mu_n)\to T(\nu)\) uniformly over all configurations in \(B_n\).
Hence, for all sufficiently large \(n\),
\(|T(\widehat\mu_n)-T(\nu)|\le\eta\).

The eigenvalue density~\citep{hiai2000semicircle} is as shown in Section~\ref{sec:prelim}:
\[
    p_n(\xi_1,\ldots,\xi_{q_n})
    =
    \frac1{Z_n}
    \prod_{i=1}^{q_n} \xi_i^{(r_n-1)/2}e^{-n\xi_i/2}
    \prod_{1\le i<j\le q_n}|\xi_i-\xi_j|,
    \qquad \xi_i>0,
\]
where \(r_n=N_n-q_n\). Therefore
\[
\mathbb P(B_n)
    =
    \frac1{Z_n}
    \int_{B_n}
    \prod_{i=1}^{q_n} \xi_i^{(r_n-1)/2}e^{-n\xi_i/2}
    \prod_{1\le i<j\le q_n}|\xi_i-\xi_j|
    \,d\xi_1\cdots d\xi_{q_n} .
\]
For configurations in \(B_n\), the empirical measures converge uniformly to
\(\nu\), so
\[
    -\frac1{n^2}\frac n2\sum_{i=1}^{q_n}\xi_i
    \to
    -\frac{\kappa}{2}\int u\,d\nu(u), \text{ and }
    \frac1{n^2}\frac{r_n-1}{2}\sum_{i=1}^{q_n}\log\xi_i
    \to
    \frac{\kappa\rho}{2}\int\log u\,d\nu(u).
\]
Both convergences are uniform on \(B_n\). The uniform-grid construction also gives
\[
    \frac1{n^2}\sum_{1\le i<j\le q_n}\log|\xi_i-\xi_j|
    \to
    \frac{\kappa^2}{2}
    \iint\log|u-v|\,d\nu(u)d\nu(v),
\]
uniformly on \(B_n\).

The volume of \(B_n\) satisfies $\prod_{\ell,j}|B_{\ell,j,n}|
    \ge
    \exp\{-Cq_n\log n\}
    =
    \exp\{o(n^2)\}.$ Finally, the Laguerre--Selberg asymptotics~\citep{hiai1998eigenvalue, hiai2000semicircle} for the partition function, with the
normalization in \eqref{eq:wishart-rate-prelim}, give
\(-\lim_{n\to\infty}n^{-2}\log Z_n=C_{\kappa,\rho}\).
Combining these estimates with the definition of \(I_{\kappa,\rho}\), we get
$
    \lim_{n\to\infty}
    \frac1{n^2}\log\mathbb P(B_n)
    =
    -I_{\kappa,\rho}(\nu).
$
This proves Lemma~\ref{lem:local-box-event-piecewise-left}.
\end{proof}

Apply Lemma~\ref{lem:local-box-event-piecewise-left} to this \(\nu\) and
\(\eta\), and notice that the constructed local box event satisfies \(B_n\subseteq \{H_n<x\}\) for large enough \(n\). Indeed, on \(B_n\), \(T(\widehat\mu_n)\le T(\nu)+\eta<x/\kappa-\eta\).
Since \(H_n=\kappa_nT(\widehat\mu_n)\) with
\(\kappa_n:=q_n/n\to\kappa\), for all sufficiently large \(n\),
\(\kappa_n(x/\kappa-\eta)<x\). Thus
\(B_n\subseteq\{H_n<x\}\) for all sufficiently large \(n\). Consequently,
\[
\begin{aligned}
    \liminf_{n\to\infty}
    \frac1{n^2}\log\mathbb P(H_n<x)
    &\ge
    \lim_{n\to\infty}
    \frac1{n^2}\log\mathbb P(B_n) =
    -I_{\kappa,\rho}(\nu)\ge
    -\mathcal J_{\kappa,\rho}(x/\kappa)-2\varepsilon,
\end{aligned}
\]
where we used \(I_{\kappa,\rho}(\nu)\le I_{\kappa,\rho}(\mu)+\varepsilon\) and
\(I_{\kappa,\rho}(\mu)\le \mathcal J_{\kappa,\rho}(x/\kappa)+\varepsilon\).
Sending \(\varepsilon\downarrow0\) proves the lower bound.
\end{proof}

Combining Proposition~\ref{prop:left-tail-upper} and
Proposition~\ref{prop:left-tail-lower}, we prove
\eqref{eq:gaussian-left-tail-asymptotic}.

\begin{remark}[Typical and impossible regimes]
\label{rem:left-tail-typical-impossible}
Let \(\mu_{\kappa,\rho}\) denote the Marchenko--Pastur minimizer of
\(I_{\kappa,\rho}\). Since
\(h_{\kappa,\rho}=\kappa\int u^{-1}\,d\mu_{\kappa,\rho}(u)\), if
\(x\ge h_{\kappa,\rho}\), then \(\mu_{\kappa,\rho}\) is feasible for the
variational problem and therefore \(\mathcal J_{\kappa,\rho}(x/\kappa)=0\).
Thus the estimate gives the trivial exponent zero above the typical value.

If \(x<0\), then \(x/\kappa<0\), while \(T(\mu)>0\) for every probability
measure on \((0,\infty)\). In this case the variational problem is empty and
the natural convention is \(\mathcal J_{\kappa,\rho}(x/\kappa)=+\infty\).
Indeed, since \(H_n=\kappa_nT(\widehat\mu_n)>0\), the event
\(\{H_n\le x\}\) is empty whenever \(x<0\). The boundary case \(x=0\) is not
needed for the left-tail result.
\end{remark}

\begin{remark}[Why continuity of the variational rate is enough]
\label{rem:left-tail-continuity-enough}
The only possible mismatch between the upper and lower bounds is that the upper
bound naturally gives the closed constraint \(T(\mu)\le x/\kappa\), whereas the
local box event lower bound uses the strict slack \(T(\mu)<x/\kappa\).
Lemma~\ref{lem:left-variational-rate-continuity} removes this mismatch:
\(\inf_{T(\mu)\le x/\kappa}I_{\kappa,\rho}(\mu)=
\inf_{T(\mu)<x/\kappa}I_{\kappa,\rho}(\mu)\).
The key observation is the scaling perturbation
\(\mu^{(c)}=(u\mapsto cu)_\#\mu\), \(c>1\), which strictly decreases \(T\)
while changing \(I_{\kappa,\rho}\) continuously as \(c\downarrow1\).
\end{remark}

\section{\texorpdfstring{\(n^2\)-speed}{n-squared-speed} right-tail suppression for ridge regression}
\label{sec:ridge-right-tail}

We next turn from ridgeless to ridge regression, to identify what protection is lost when the ridge penalty is removed. The preceding sections show that ridgeless risk has a sharply asymmetric tail profile: its left tail is an \(n^2\)-scale collective spectral large deviation, whereas its right tail is driven by a microscopic hard-edge event and decays only at the \(n\log n\) scale. For fixed \(\lambda>0\), ridge replaces the singular weight \(u^{-1}\) by bounded regularized weights. A single near-zero eigenvalue can therefore no longer create an order-one increase in risk. Fixed upward deviations must instead arise from a collective deformation of the empirical spectrum and hence occur at speed \(n^2\).




By the risk formula in \eqref{eq:ridge-risk-equivalent-prelim}, the population risk for ridge regression~\citep{dobriban2018high} in
both regimes can be written uniformly as
\begin{equation}
\label{eq:ridge-risk-positive-spectrum}
    R_{n,\lambda}
    =
    b_n+\frac1n\sum_{i=1}^{q_n}
    \varphi_{\lambda,n}(\lambda_{i,n}), \text{ where } \varphi_{\lambda,n}(u)
    :=
    \frac1{u+\lambda}
    +
    \frac{\lambda(\lambda\alpha^2-\gamma_n)}{\gamma_n}
    \frac1{(u+\lambda)^2}.
\end{equation}
Here \(b_n=\alpha^2(1-\gamma_n^{-1})_+\), same as in
\eqref{eq:ridgeless-risk-prelim}.

Recall that \(\mu_\gamma:=\mu_{\kappa_\gamma,\rho_\gamma}\) and
\(I_\gamma:=I_{\kappa_\gamma,\rho_\gamma}\). We define the typical ridge risk
    $r_{\lambda,\star}
    :=
    b_\gamma
    +
    \kappa_\gamma
    \int \varphi_{\lambda,\gamma}(u)\,d\mu_\gamma(u)$.

\begin{theorem}[High-dimensional LDP for ridge regression with Gaussian data]
\label{thm:gaussian-ridge-risk-ldp}
Assume Assumption~\ref{ass:gaussian-entries}. For every fixed
\(\lambda>0\), the sequence \((R_{n,\lambda})_{n\ge1}\) satisfies an LDP on
\(\mathbb R\) with speed \(n^2\) and good rate function
    $J_{\gamma,\lambda}(y)
    :=
    \inf\left\{
        I_\gamma(\mu):
        b_\gamma
        +
        \kappa_\gamma
        \int \varphi_{\lambda,\gamma}(u)\,d\mu(u)
        =
        y
    \right\}$.
In particular, for every fixed \(\delta>0\),
\begin{equation}
\label{eq:gaussian-ridge-right-tail-upper}
    \limsup_{n\to\infty}
    \frac1{n^2}
    \log\mathbb P(R_{n,\lambda}>r_{\lambda,\star}+\delta)
    \le
    -J_{\gamma,\lambda}^{+}(\delta),
\end{equation}
where
    $J_{\gamma,\lambda}^{+}(\delta)
    :=
    \inf\left\{
        I_\gamma(\mu):
        b_\gamma
        +
        \kappa_\gamma
        \int \varphi_{\lambda,\gamma}(u)\,d\mu(u)
        \ge
        r_{\lambda,\star}+\delta
    \right\}
    >0$.
\end{theorem}

\begin{proof}
Since \(\lambda>0\), the function
\(\varphi_{\lambda,\gamma}\) is bounded and continuous on \(\mathbb R_+\). Hence
the map \(\mu\mapsto b_\gamma+\kappa_\gamma
\int \varphi_{\lambda,\gamma}(u)\,d\mu(u)\) is continuous for the weak topology.
The Gaussian Wishart spectral LDP and the contraction principle therefore imply
an \(n^2\)-speed LDP for
\(b_\gamma+\kappa_\gamma\int \varphi_{\lambda,\gamma}(u)\,d\widehat\mu_n(u)\)
with rate function \(J_{\gamma,\lambda}\).

It remains only to replace this limiting functional by the actual ridge risk.
By \eqref{eq:ridge-risk-positive-spectrum},
    $R_{n,\lambda}
    =
    b_n
    +
    \frac{q_n}{n}
    \int \varphi_{\lambda,n}(u)\,d\widehat\mu_n(u)$.
Since \(b_n\to b_\gamma\), \(q_n/n\to\kappa_\gamma\), and
\(\varphi_{\lambda,n}\to\varphi_{\lambda,\gamma}\) uniformly on
\(\mathbb R_+\), we have the deterministic bound
    $\sup
    \left|
        R_{n,\lambda}
        -
        \left(
        b_\gamma
        +
        \kappa_\gamma
        \int \varphi_{\lambda,\gamma}(u)\,d\widehat\mu_n(u)
        \right)
    \right|
    \to0$ as $n\to\infty$.
Thus the two sequences are exponentially equivalent at speed \(n^2\), and
\(R_{n,\lambda}\) has the same LDP.

The right-tail upper bound follows by applying the LDP upper bound to the closed
set \([r_{\lambda,\star}+\delta,\infty)\). Finally,
\(J_{\gamma,\lambda}^{+}(\delta)>0\), because \(I_\gamma\) has its unique zero
at \(\mu_\gamma\), and
\(b_\gamma+\kappa_\gamma\int\varphi_{\lambda,\gamma}\,d\mu_\gamma
=r_{\lambda,\star}\).
\end{proof}

\begin{proposition}[Right-tail upper bound for ridge regression under general entries]
\label{prop:general-ridge-right-tail-upper}
Assume the normalization and log-Sobolev conditions B.1 and B.4 in
Assumption~\ref{ass:general-entries}. For every fixed \(\lambda>0\) and
\(\delta>0\), there exists \(c_{\lambda,\delta}>0\) such that, for all
sufficiently large \(n\),
\begin{equation}
\label{eq:general-ridge-right-tail-upper}
    \mathbb P(R_{n,\lambda}>r_{\lambda,\star}+\delta)
    \le
    \exp\{-c_{\lambda,\delta}n^2\}.
\end{equation}
Consequently $  \limsup_{n\to\infty}
    \frac1{n^2}
    \log\mathbb P(R_{n,\lambda}>r_{\lambda,\star}+\delta)
    <0.$
\end{proposition}

\begin{proof}
The functions \(\varphi_{\lambda,n}\) satisfy
\(\|\varphi_{\lambda,n}\|_\infty+\|\varphi_{\lambda,n}\|_{\mathrm{Lip}}
=O_\lambda(1)\), because \(\lambda>0\) is fixed and \(\gamma_n\to\gamma\ne0\).
By the universality of Marchenko--Pastur law under condition B.1~\citep{bai2010spectral},
bounded convergence theorem,  and the uniform convergence
\(\varphi_{\lambda,n}\to\varphi_{\lambda,\gamma}\), we have $\mathbb E R_{n,\lambda}
    =
    b_n
    +
    \mathbb E\left[
        \frac1n\sum_{i=1}^{q_n}
        \varphi_{\lambda,n}(\lambda_{i,n})
    \right]
    \to
    r_{\lambda,\star}.$ Hence, for all sufficiently large \(n\),
\(\mathbb E R_{n,\lambda}\le r_{\lambda,\star}+\delta/2\).
Therefore
\[
    \mathbb P(R_{n,\lambda}>r_{\lambda,\star}+\delta)
    \le
    \mathbb P\left(
        R_{n,\lambda}-\mathbb E R_{n,\lambda}>
        \frac{\delta}{2}
    \right)=\mathbb P\left(
        \frac1n\sum_{i=1}^{q_n}
        \varphi_{\lambda,n}(\lambda_{i,n})
        -
        \mathbb E\left[
        \frac1n\sum_{i=1}^{q_n}
        \varphi_{\lambda,n}(\lambda_{i,n})
        \right]
        >
        \frac{\delta}{2}
    \right).
\]
For each \(n\), apply
Lemma~\ref{lem:lipschitz-spectral-statistic-concentration} with
\(h=\varphi_{\lambda,n}\) and \(t=\delta/2\). Since
\(\|\varphi_{\lambda,n}\|_\infty+\|\varphi_{\lambda,n}\|_{\mathrm{Lip}}
=O_\lambda(1)\), for all sufficiently large \(n\) the lower-threshold condition in Lemma~\ref{lem:lipschitz-spectral-statistic-concentration} is satisfied, and hence

\[
    \mathbb P(R_{n,\lambda}>r_{\lambda,\star}+\delta)
    \le
    e^{-cn^2}
    +
    \exp\{-cn^2\delta^2\}
    \le
    \exp\{-c_{\lambda,\delta}n^2\}
\]
for all sufficiently large \(n\). This proves the claim.
\end{proof}

\begin{remark}[Ridge versus ridgeless right tails]
The essential difference is that \(\lambda>0\) removes the singularity at the
hard edge. The ridge statistic uses the bounded functions
\((u+\lambda)^{-1}\) and \((u+\lambda)^{-2}\), so its right tail is controlled
by concentration of bounded spectral statistics. In contrast, the ridgeless
statistic uses \(u^{-1}\), so a single eigenvalue at scale \(n^{-1}\) already
creates an \(O(1)\) upward deviation, leading to the \(n\log n\) speed.
\end{remark}

\section{Beyond Ridgeless Regression: Algorithm-Independent Lower Bounds for Linear Interpolators}
\label{sec:beyond-ridgeless}

Exact interpolation forces the row-space component of the estimation error to equal \(X^\top(XX^\top)^{-1}\varepsilon\), regardless of how the coefficient vector is selected. Combining this deterministic lower bound with the preceding hard-edge estimates yields the algorithm-independent tail lower bounds below.

Throughout this section, we work in the overparameterized exact-interpolation
setting \(p>n\). Let \(y=X\beta+\varepsilon\), where
\(X\in\mathbb R^{n\times p}\), \(p>n\), and
\(\varepsilon\sim N(0,I_n)\). Let \(x_0\) be an independent test covariate with
\(\mathbb E x_0x_0^\top=I_p\). We call a possibly randomized algorithm \(A\)
an exact linear interpolating algorithm if, with internal randomness \(U\)
independent of \((X,\beta,\varepsilon)\), it outputs a coefficient vector
\(\widehat\beta_A=A(X,y,U)\in\mathbb R^p\) satisfying
\(X\widehat\beta_A=y\) almost surely. For fixed \(X\) and fixed \(\beta\), define
the conditional risk of this exact interpolator by
\[
    R_A(X,\beta)
    :=
    \mathbb E_{\varepsilon,U,x_0}
    \left[
        \{x_0^\top(\widehat\beta_A-\beta)\}^2
        \mid X,\beta
    \right].
\]
Because \(\mathbb E x_0x_0^\top=I_p\),
\[
    R_A(X,\beta)
    =
    \mathbb E_{\varepsilon,U}
    \left[
        \|\widehat\beta_A-\beta\|_2^2
        \mid X,\beta
    \right].
\]
No random-effects assumption on \(\beta\) is imposed in this section.

The following lemma identifies the inverse-trace lower bound shared by all
exact linear interpolating algorithms.

\begin{lemma}[Interpolator risk lower bound]
\label{lem:exact-interpolator-row-space-lower-bound}
Let \(X\in\mathbb R^{n\times p}\), \(p>n\), have full row rank, and let
\(\widehat\beta_A\) be any exact interpolator of \(y=X\beta+\varepsilon\).
Then
    $R_A(X,\beta)\ge \operatorname{Tr}\{(XX^\top)^{-1}\}$.
Equivalently, if \(\lambda_1,\ldots,\lambda_n\) are the eigenvalues of
\(XX^\top/n\), then
\(R_A(X,\beta)\ge n^{-1}\sum_{i=1}^n\lambda_i^{-1}\).
\end{lemma}

\begin{proof}
Set \(h_A:=\widehat\beta_A-\beta\), and let
\(P_X:=X^\top(XX^\top)^{-1}X\) be the orthogonal projector onto
\(\operatorname{row}(X)\). Exact interpolation gives
\(Xh_A=X\widehat\beta_A-X\beta=\varepsilon\). Since
\(P_Xh_A\in\operatorname{row}(X)\) and
\(X:\operatorname{row}(X)\to\mathbb R^n\) is one-to-one and onto, the unique
row-space vector mapped to \(\varepsilon\) is
\(X^\top(XX^\top)^{-1}\varepsilon\). Thus
\(P_Xh_A=X^\top(XX^\top)^{-1}\varepsilon\).

By the orthogonal decomposition into row space and null space,
\[
    \|h_A\|_2^2
    =
    \|P_Xh_A\|_2^2+\|(I-P_X)h_A\|_2^2
    \ge
    \varepsilon^\top(XX^\top)^{-1}\varepsilon .
\]
Taking conditional expectation over \(\varepsilon\) and \(U\) gives
\(R_A(X,\beta)\ge \operatorname{Tr}\{(XX^\top)^{-1}\}\). If
\(\lambda_1,\ldots,\lambda_n\) are the eigenvalues of \(XX^\top/n\), then
\(\operatorname{Tr}\{(XX^\top)^{-1}\}=n^{-1}\sum_{i=1}^n\lambda_i^{-1}\), which
proves the final claim.
\end{proof}

Combining the interpolator risk lower bound in
Lemma~\ref{lem:exact-interpolator-row-space-lower-bound} with the
fixed-dimensional hard-edge estimate from Section~\ref{sec:fixed-dimensional-tail}
gives the following consequence.

\begin{theorem}[Fixed-dimensional interpolator risk heavy tail]
\label{thm:exact-interpolator-fixed-lower}
Fix \(n<p\). Suppose the entries of \(X\in\mathbb R^{n\times p}\) satisfy
the nondegeneracy-near-zero condition B.3 in
Assumption~\ref{ass:general-entries}. Let \(A\) be any possibly randomized
exact interpolating algorithm. Then, for every fixed
\(\beta\in\mathbb R^p\), there exist constants \(c>0\) and \(x_0<\infty\) such
that, for all \(x\ge x_0\),
\[
    \mathbb P(R_A(X,\beta)>x)
    \ge
    c x^{-(p-n+1)/2}.
\]
\end{theorem}

\begin{proof}
The entries have densities, so \(X\) has full row rank almost surely. By
Lemma~\ref{lem:exact-interpolator-row-space-lower-bound},
\(R_A(X,\beta)\ge \operatorname{Tr}\{(XX^\top)^{-1}\}\). Applying the
lower-bound argument in the proof of
Theorem~\ref{thm:fixed-dimensional-ridgeless-heavy-tail} with \(G=X^\top\),
\(N=p\), \(q=n\), and \(r=p-n\), equivalently using
Proposition~\ref{prop:lower-small-ball-distance-subspace} for one column
conditioned on the span of the others, gives
\[
    \mathbb P(\operatorname{Tr}\{(XX^\top)^{-1}\}>x)
    \ge
    c x^{-(p-n+1)/2}
\]
for all sufficiently large \(x\). The desired bound follows.
\end{proof}

The same reduction, combined with the microscopic smallest-eigenvalue lower bound from Section~\ref{sec:proportional-right-tail}, gives the proportional analogue.

\begin{theorem}[Large deviation lower bound for the right tail of interpolator]
\label{thm:exact-interpolator-proportional-lower}
Let \(p_n/n\to\gamma>1\). Suppose the entries of
\(X_n\in\mathbb R^{n\times p_n}\) satisfy the nondegeneracy-near-zero condition
B.3 in Assumption~\ref{ass:general-entries}. Let \(A_n\) be any sequence of
possibly randomized exact interpolating algorithms, and let
\(\beta_n\in\mathbb R^{p_n}\) be arbitrary deterministic vectors. Then, for
every fixed \(t>0\),
\[
    \liminf_{n\to\infty}
    \frac{1}{n\log n}
    \log\mathbb P(R_{A_n}(X_n,\beta_n)>t)
    \ge
    -\frac{\gamma-1}{2}.
\]
\end{theorem}

\begin{proof}
Let \(0<\lambda_{1,n}\le\cdots\le\lambda_{n,n}\) be the eigenvalues of
\(X_nX_n^\top/n\). By
Lemma~\ref{lem:exact-interpolator-row-space-lower-bound},
\[
    R_{A_n}(X_n,\beta_n)
    \ge
    H_n:=\frac1n\sum_{i=1}^n\lambda_{i,n}^{-1}.
\]
As in Proposition~\ref{prop:right-tail-lower}, choose \(c>0\) such that \(c^{-1}>t\). On the event
\(\lambda_{1,n}\le c/n\), we have
\(H_n\ge (n\lambda_{1,n})^{-1}\ge c^{-1}>t\). Hence
    $\{\lambda_{1,n}\le c/n\}
    \subseteq
    \{R_{A_n}(X_n,\beta_n)>t\}$.
Applying Lemma~\ref{lem:microscopic-smallest-eigenvalue-lower} with
\(G_n=X_n^\top\), \(N_n=p_n\), \(q_n=n\), and \(\rho=\gamma-1\) gives the claim.
\end{proof}

\begin{remark}[Relation to interpolation-set geometry]
\cite{chen2026abundant} study a complementary geometric
question about interpolation. Their object is the set of unit-norm linear
classifiers that interpolate a fixed labeled dataset. They analyze the
generalization error of a classifier sampled uniformly from this interpolation
set and derive quenched large-deviation principles describing the
exponential-scale volume of interpolators with a prescribed performance.

The results in this section concern a different object. We study noisy linear
regression and prove lower bounds that hold for any exact linear interpolating
algorithm. Thus we do not measure the volume of good interpolators inside an
interpolation set; instead, we identify a risk lower bound that every exact
linear interpolator must inherit from fitting noisy labels. These two viewpoints
are complementary: one concerns the geometry of the interpolation set in
classification, while the other concerns algorithm-independent risk lower
bounds for exact interpolation in noisy regression.
\end{remark}

\section{Numerical experiments}
\label{sec:numerics}

We test the fixed-dimensional polynomial tail in Theorem~\ref{thm:fixed-dimensional-risk-tail} at moderate dimensions. The experiments examine the predicted exponent, the single-eigenvalue mechanism, and its persistence across Gaussian, Laplace, and uniform designs.

We isolate the random spectral statistic that appears in the exact ridgeless risk representation. For fixed integers \(N\ge q\), let \(G\in\mathbb R^{N\times q}\) have independent entries and define $ Y=\operatorname{Tr}\{(G^\top G)^{-1}\}.$ Theorem~\ref{thm:fixed-dimensional-risk-tail} predicts the polynomial tail exponent through
\[
    \mathbb P(Y>x)=x^{-\alpha+o(1)}, \text{ for } x\to\infty, 
    \text{ with }
    \alpha=\frac{N-q+1}{2}.
\]
Thus the codimension \(r=N-q\) determines the rate at which rare large risk values disappear. From an operations perspective, this exponent is the quantity that governs the frequency of extreme estimation errors under repeated deployment, stress testing, or simulation-based calibration.

We use two complementary diagnostics. First, for simulated samples \(Y_1,\ldots,Y_m\), we estimate the empirical logarithmic survival exponent $\frac{\log \widehat{\mathbb P}(Y>x)}{\log x}, \text{ where }
    \widehat{\mathbb P}(Y>x)=\frac1m\sum_{i=1}^m \mathbf 1\{Y_i>x\}.$ Under the prediction above, this quantity should approach \(-\alpha\) in the far right tail. Second, we compute the Hill estimator~\citep{hill1975simple} on the largest order statistics. If \(Y_{(1)}\ge \cdots \ge Y_{(m)}\) denote the decreasing order statistics, then $\widehat \xi_k=\frac1{k}\sum_{j=1}^k
    \log \frac{Y_{(j)}}{Y_{(k+1)}} .$ For a tail satisfying \(\mathbb P(Y>x)\sim Cx^{-\alpha}\), the extreme value tail index is \(\xi=1/\alpha\), so the corresponding exponent estimate is \(-1/\widehat\xi_k\). The two diagnostics use different parts of the simulated tail and therefore provide a useful check against artifacts from threshold selection.

Table~\ref{tab:distribution-figure-table} summarizes the exponent diagnostics. For Gaussian entries, the fitted log-log slopes are \(-0.495\), \(-1.014\), \(-2.105\), and \(-3.290\), while the theoretical values are \(-0.5\), \(-1\), \(-2\), and \(-3\). The agreement is strongest for the heavier tails corresponding to smaller \(r\), where the relevant tail region is well represented with the available Monte Carlo budget. The case \(r=5\) is more demanding because the target tail is substantially steeper; resolving its far tail requires many more extreme observations. The Laplace and uniform rows show the same ordering of exponents and the same qualitative alignment with the hard-edge prediction, indicating that the observed effect is not an artifact of the Gaussian Wishart specification.

To verify the structural mechanism behind the right tail, we also compute the fraction of \(Y\) contributed by the smallest eigenvalue,
\[
    \frac{1/\lambda_{\min}(G^\top G)}
    {\operatorname{Tr}\{(G^\top G)^{-1}\}} .
\]
The proof of the fixed-dimensional tail result predicts that extreme values of \(Y\) arise from a single near-dependence event, equivalently from one eigenvalue approaching the hard edge. Figure~\ref{fig:fixed-dim-aux} is consistent with this mechanism. Conditional on increasingly severe upper tail events, the median contribution of the smallest eigenvalue moves close to one across all three entry distributions. Hence the numerical evidence supports not only the exponent predicted by Theorem~\ref{thm:fixed-dimensional-risk-tail}, but also the microscopic spectral mechanism responsible for extreme ridgeless risk.

\begin{table}[!htbp]
\centering
\caption{
Finite-sample diagnostics for the fixed-dimensional right tail of
\(Y=\operatorname{Tr}\{(G^\top G)^{-1}\}\).
The experiment uses \(q=6\), \(r=N-q\in\{0,1,3,5\}\), and \(2\times 10^5\) independent replications for each pair of entry distribution and value of \(r\).
The rows correspond to Gaussian, Laplace, and uniform entries.
The middle column reports tail exponent diagnostics, including the empirical logarithmic survival exponent and the Hill-based estimate \(-1/\widehat \xi_k\).
The right column reports the empirical log-log survival curves.
Horizontal reference lines mark the theoretical exponents \(-(r+1)/2\).
}
\label{tab:distribution-figure-table}

\setlength{\tabcolsep}{3pt}
\renewcommand{\arraystretch}{1.15}

\newlength{\widefig}
\newlength{\narrowfig}
\setlength{\widefig}{0.52\textwidth}
\setlength{\narrowfig}{0.26\textwidth}

\begin{tabular}{@{}>{\raggedright\arraybackslash}p{0.13\textwidth}cc@{}}
\toprule
Entry distribution
& Tail exponent diagnostics
& Empirical log log survival curves
\\
\midrule

Gaussian
&
\includegraphics[width=\widefig]{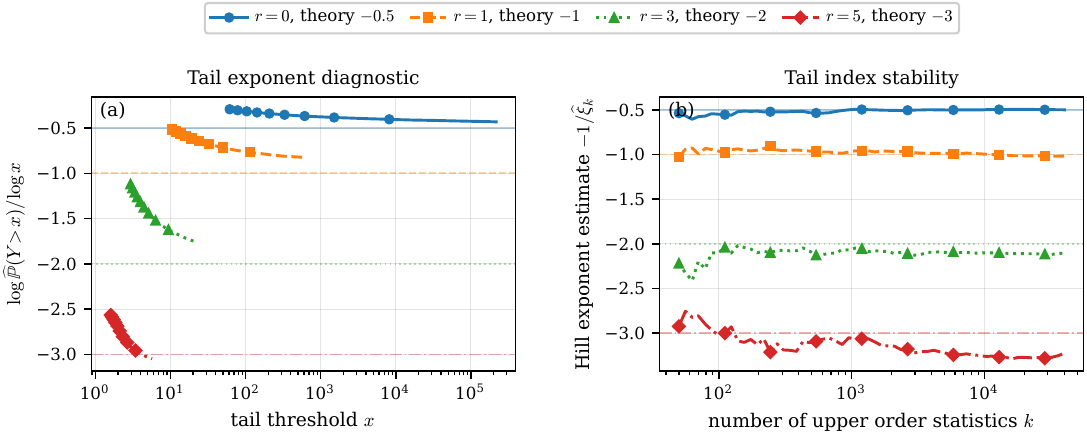}
&
\includegraphics[width=\narrowfig]{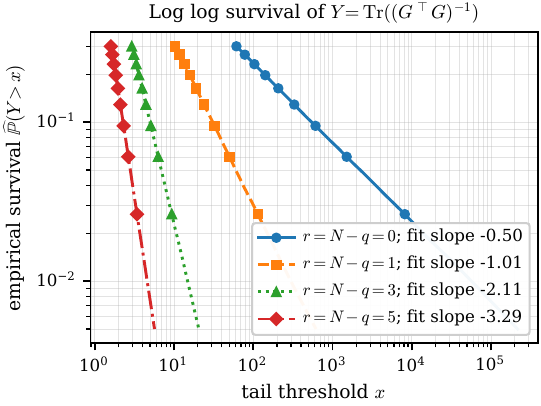}
\\[0.8em]

Laplace
&
\includegraphics[width=\widefig]{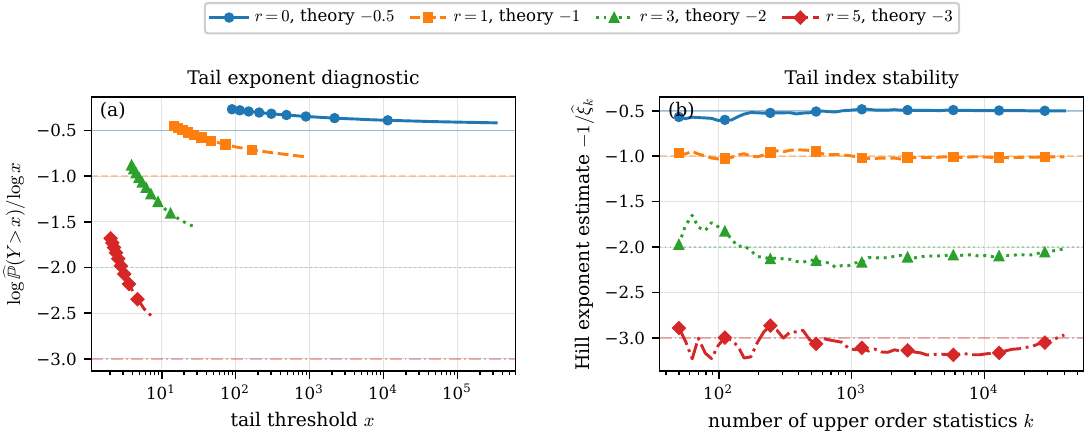}
&
\includegraphics[width=\narrowfig]{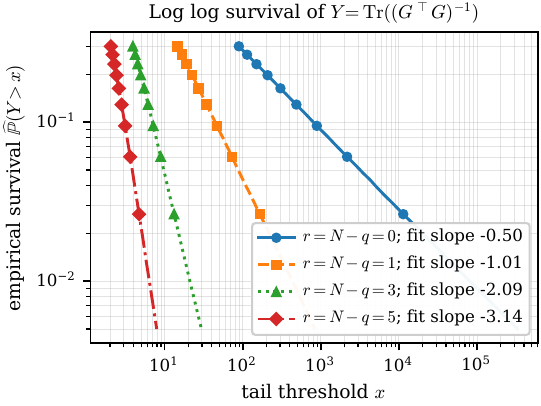}
\\[0.8em]

Uniform
&
\includegraphics[width=\widefig]{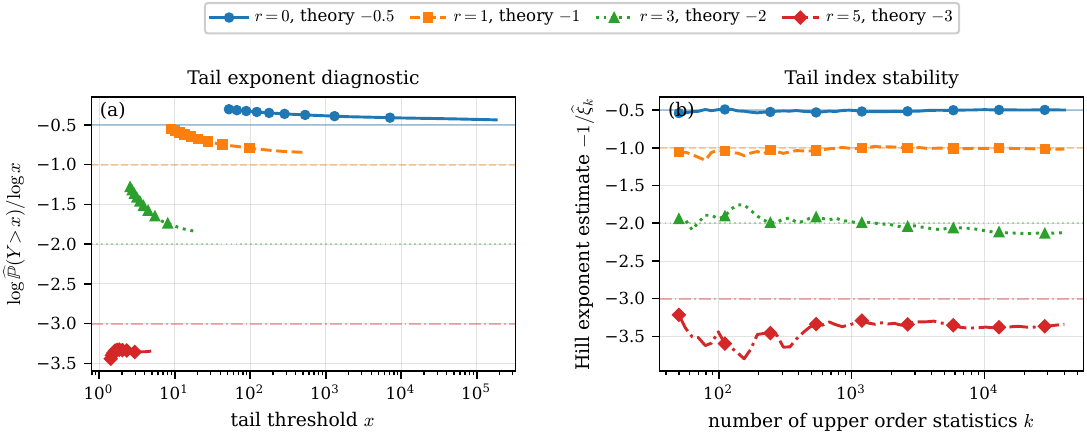}
&
\includegraphics[width=\narrowfig]{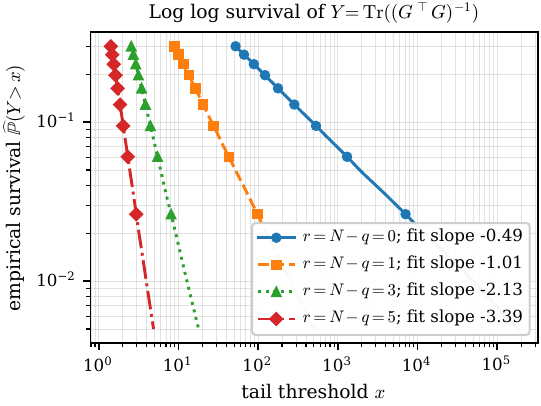}
\\

\bottomrule
\end{tabular}
\end{table}

\begin{figure}[!htbp]
\centering
\begin{minipage}[t]{0.32\textwidth}
\centering
\includegraphics[width=\linewidth]{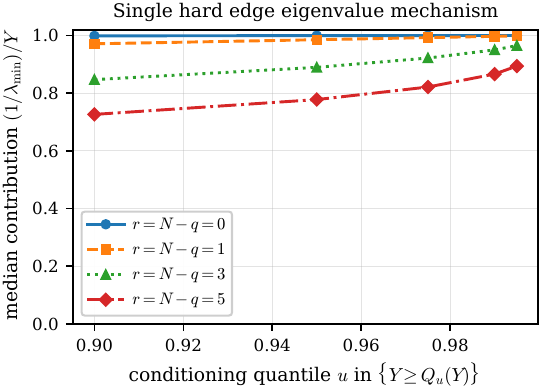}
\smallskip
\centerline{(a) Gaussian entries}
\end{minipage}\begin{minipage}[t]{0.32\textwidth}
\centering
\includegraphics[width=\linewidth]{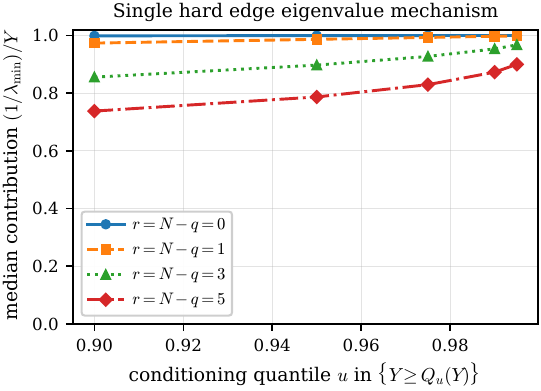}
\smallskip
\centerline{(b) Laplace entries}
\end{minipage}\begin{minipage}[t]{0.32\textwidth}
\centering
\includegraphics[width=\linewidth]{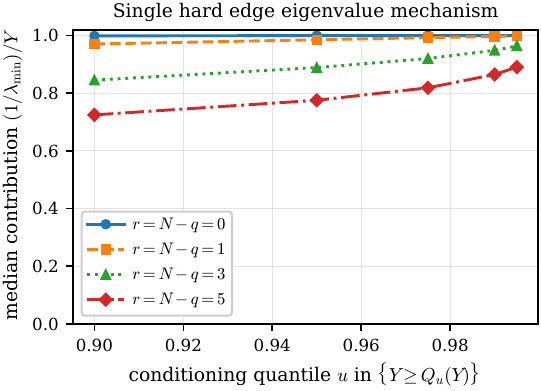}
\smallskip
\centerline{(c) Uniform entries}
\end{minipage}
\caption{
Smallest eigenvalue contribution to the inverse trace, conditional on \(Y\) exceeding an upper empirical quantile.
Each curve reports the median value of
\((1/\lambda_{\min}(G^\top G))/\operatorname{Tr}\{(G^\top G)^{-1}\}\) under the corresponding conditioning event.
Values close to one indicate that extreme risk realizations are dominated by a single hard edge eigenvalue.
}
\label{fig:fixed-dim-aux}
\end{figure}

\section{Proofs of auxiliary estimates}
\label{sec:proofs-auxiliary-estimates}

\subsection{Fixed-dimensional auxiliary estimates}

\begin{proof}[Proof of Proposition~\ref{cor:distance-subspace-upper-small-ball}.]
Take \(E=H^\perp\), so \(\dim(E)=d\) and
\(\operatorname{dist}(Z,H)=\|P_{H^\perp}Z\|_2\). The bounded-density
projection estimate of \citet[Theorem~2.5]{nguyen2018random} gives, for
orthogonal projections onto \(d\)-dimensional subspaces,
\[
    \mathbb P(\|P_E Z-z\|_2\le a)
    \le
    \left(
        \frac{C K_{\rm up}a}{\sqrt d}
    \right)^d ,
\]
for every \(z\in E\) and \(a\ge0\) (see also Theorem~1.1 and Corollary~1.4
of~\cite{rudelson2015small}). Applying this estimate with \(z=0\) and
\(a=t\) gives the stated bound.
\end{proof}

\begin{proof}[Proof of Proposition~\ref{cor:distance-subspace-lower-small-ball}.]
Take \(E=H^\perp\). Then \(\dim(E)=d\) and
\(\operatorname{dist}(X,H)=\|P_E X\|_2\). Let \(p_\xi\) denote the common
density of the entries of \(X\). By the nondegeneracy-near-zero condition in
Assumption~\ref{ass:general-entries}, \(p_\xi(x)\ge k_{\rm low}\) for
\(|x|\le u_0\).

Write \(x=y+z\), where \(y\in E\) and \(z\in E^\perp\). If
\(\|y\|_2\le a\le u_0/2\), then \(\|y\|_\infty\le u_0/2\). If also
\(z\in E^\perp\cap[-u_0/2,u_0/2]^N\), then \(\|z\|_\infty\le u_0/2\), hence
\(|(y+z)_i|\le u_0\) for \(1\le i\le N\), and on this set
\(\prod_{i=1}^N p_\xi((y+z)_i)\ge k_{\rm low}^N\).
Using the orthogonal decomposition \(\mathbb R^N=E\oplus E^\perp\),
\[
\begin{aligned}
    \mathbb P(\|P_E X\|_2\le a)
    &\ge
    k_{\rm low}^N
    \operatorname{Vol}_d(B_2^d(a))
    \operatorname{Vol}_{N-d}
    \bigl(E^\perp\cap[-u_0/2,u_0/2]^N\bigr).
\end{aligned}
\]
By Vaaler's theorem for central sections of the cube
\citep{rogers1958packing,vaaler1979geometric},
\[
    \operatorname{Vol}_{N-d}
    \bigl(E^\perp\cap[-u_0/2,u_0/2]^N\bigr)
    \ge u_0^{N-d}.
\]
Thus
\[
    \mathbb P(\|P_E X\|_2\le a)
    \ge
    k_{\rm low}^N u_0^{N-d}\operatorname{Vol}_d(B_2^d(a)).
\]
For fixed \(N,d\), this is at least \(c_{\rm low}a^d\) for
\(0<a\le u_0/2\), after absorbing the positive constants into
\(c_{\rm low}\).
\end{proof}

\begin{proof}[Proof of Lemma~\ref{lem:negative-second-moment-fixed}.]
It suffices to prove that the \(i\)-th diagonal entry of \(W^{-1}\) is
\(D_i^{-2}\). After permuting columns, take \(i=q\) and write
\(G=[G_{-q},g_q]\). Then
\[
    W=
    \begin{pmatrix}
        G_{-q}^\top G_{-q} & G_{-q}^\top g_q\\
        g_q^\top G_{-q} & g_q^\top g_q
    \end{pmatrix}.
\]
By the block inverse formula, the last diagonal entry of \(W^{-1}\) is the
inverse of the Schur complement
\[
    g_q^\top g_q
    -
    g_q^\top G_{-q}(G_{-q}^\top G_{-q})^{-1}G_{-q}^\top g_q
    =
    \|(I-P_{H_q})g_q\|_2^2
    =
    D_q^2.
\]
Thus \((W^{-1})_{ii}=D_i^{-2}\) for every \(i\). Summing the diagonal entries
gives the identity.
\end{proof}

\subsection{Right-tail auxiliary estimates}

Throughout this subsection, for a matrix \(B\in\mathbb R^{N\times q}\) write
\(A_B:=n^{-1}B^\top B\), let
\(\lambda_1(B)\le\cdots\le\lambda_q(B)\) denote the eigenvalues of \(A_B\),
and set \(F_{n,h}(B):=\omega_n\sum_i h(\lambda_i(B))\).

\begin{lemma}[Local Lipschitz bound for covariance spectral statistics]
\label{lem:local-lipschitz-spectral-statistic}
Let \(G,H\in\mathbb R^{N\times q}\) and let
\(h:\mathbb R_+\to\mathbb R\) be Lipschitz. Then
\[
    |F_{n,h}(G)-F_{n,h}(H)|
    \le
    \omega_n \|h\|_{\mathrm{Lip}}
    \frac{\|G\|_F+\|H\|_F}{n}
    \|G-H\|_F .
\]
In particular, if \(\omega_n\le C_\omega/n\) and
\(\|G\|_F,\|H\|_F\le Rn\), then
\[
    |F_{n,h}(G)-F_{n,h}(H)|
    \le
    \frac{2C_\omega R\|h\|_{\mathrm{Lip}}}{n}\,
    \|G-H\|_F .
\]
\end{lemma}

\begin{proof}[Proof of Lemma~\ref{lem:local-lipschitz-spectral-statistic}.]
With the notation of the lemma, \(A_G=n^{-1}G^\top G\) and
\(A_H=n^{-1}H^\top H\).
By the Lidskii--Mirsky--Wielandt eigenvalue perturbation inequality
\citep{mirsky1957inequalities,li1999lidskii},
\[
\begin{aligned}
    \left|
        \sum_{i=1}^q h(\lambda_i(G))
        -
        \sum_{i=1}^q h(\lambda_i(H))
    \right|
    &\le
    \|h\|_{\mathrm{Lip}}
    \sum_{i=1}^q |\lambda_i(G)-\lambda_i(H)|  \\
    &\le
    \|h\|_{\mathrm{Lip}}
    \|A_G-A_H\|_* .
\end{aligned}
\]
Moreover,
\(A_G-A_H=n^{-1}G^\top(G-H)+n^{-1}(G-H)^\top H\). Using
\(\|UV\|_*\le \|U\|_F\|V\|_F\), we get
\(\|A_G-A_H\|_*\le n^{-1}(\|G\|_F+\|H\|_F)\|G-H\|_F\).
Multiplying by \(\omega_n\) proves the first claim. If
\(\omega_n\le C_\omega/n\) and \(G,H\in B_R\), then
\(|F_{n,h}(G)-F_{n,h}(H)|\le
2C_\omega R\|h\|_{\mathrm{Lip}}\|G-H\|_F/n\), which gives the local Lipschitz
bound on \(B_R\).
\end{proof}

\begin{lemma}[Rectangular lower-edge tail]
\label{lem:rectangular-lower-edge-tail}
Assume the normalization condition B.1 in Assumption~\ref{ass:general-entries},
and assume in addition that the entries have subgaussian tails.
There exist constants \(c_0>0\) and
\(c>0\), with \(c_0<\ell_{\kappa,\rho}/4\), such that
\[
    \mathbb P(\lambda_{1,n}<2c_0)\le e^{-cn}
\]
for all sufficiently large \(n\).
\end{lemma}

\begin{proof}[Proof of Lemma~\ref{lem:rectangular-lower-edge-tail}.]
We use the standard rectangular least singular value estimate of Theorem 1.1 of~\cite{rudelson2009smallest}:
there exist constants \(C,c>0\), depending only on the subgaussian norm, such
that for every \(\varepsilon>0\),
\[
    \mathbb P\left(
        s_{\min}(G_n)
        \le
        \varepsilon\bigl(\sqrt{N_n}-\sqrt{q_n-1}\bigr)
    \right)
    \le
    (C\varepsilon)^{N_n-q_n+1}
    +
    e^{-cN_n}.
\]
Since \(N_n,q_n\asymp n\) and \(N_n-q_n\asymp n\), there exists
\(d_{\kappa,\rho}>0\) such that, for all large \(n\),
\[
    \sqrt{N_n}-\sqrt{q_n-1}
    =
    \frac{N_n-q_n+1}{\sqrt{N_n}+\sqrt{q_n-1}}
    \ge
    d_{\kappa,\rho}\sqrt n .
\]
Choose \(\varepsilon_0>0\) small enough that \(C\varepsilon_0<1\). Then choose
\(c_0>0\) so small that
\(c_0<\ell_{\kappa,\rho}/4\) and
\(\sqrt{2c_0}\le \varepsilon_0 d_{\kappa,\rho}\). Since
\(\lambda_{1,n}=s_{\min}(G_n)^2/n\),
the event \(\{\lambda_{1,n}<2c_0\}\) implies
\[
    s_{\min}(G_n)
    <
    \sqrt{2c_0 n}
    \le
    \varepsilon_0
    \bigl(\sqrt{N_n}-\sqrt{q_n-1}\bigr).
\]
Therefore
\[
    \mathbb P(\lambda_{1,n}<2c_0)
    \le
    (C\varepsilon_0)^{N_n-q_n+1}
    +
    e^{-cN_n}
    \le
    e^{-c'n}
\]
after adjusting constants.
\end{proof}

\begin{proof}[Proof of Lemma~\ref{lem:weighted-hard-edge-pigeonhole}.]
If \(K_\eta=0\), then \(S_\eta=0\), so the event \(\{S_\eta>u\}\) is empty.
Suppose \(K_\eta\ge1\). If, for every \(1\le k\le K_\eta\),
\[
    \lambda_k\ge \frac{M_\theta k^\theta}{nu},
\]
then
\[
\begin{aligned}
    S_\eta
    &=
    \frac1n\sum_{k=1}^{K_\eta}\frac1{\lambda_k}
    \le
    \frac1n\sum_{k=1}^{K_\eta}
    \frac{nu}{M_\theta k^\theta}
    =
    u\,
    \frac{\sum_{k=1}^{K_\eta}k^{-\theta}}{M_\theta}
    \le
    u.
\end{aligned}
\]
Therefore, on \(\{S_\eta>u\}\), there exists \(1\le k\le K_\eta\) such that
\(\lambda_k<M_\theta k^\theta/(nu)\). Since \(k\le K_\eta\), we also have
\(\lambda_k<\eta\). This proves the inclusion.
\end{proof}

\begin{lemma}[Grassmannian covering]
\label{lem:grassmannian-covering}
Let \(\mathrm{Gr}(k,q)\) be the Grassmannian of \(k\)-dimensional subspaces of
\(\mathbb R^q\) with metric \(d(E,F):=\|P_E-P_F\|_{\rm op}\). For every
\(0<\varepsilon<1\), there exists an \(\varepsilon\)-net
\(\mathcal N_\varepsilon\subseteq \mathrm{Gr}(k,q)\) such that
\[
    |\mathcal N_\varepsilon|
    \le
    \left(\frac{C}{\varepsilon}\right)^{k(q-k)}.
\]
\end{lemma}

\begin{proof}[Proof of Lemma~\ref{lem:grassmannian-covering}.]
This is the standard metric entropy bound for the Grassmannian in the
projection metric~\citep{pajor1998metric}.
\end{proof}

\begin{lemma}[Alignment of nearby subspaces]
\label{lem:subspace-alignment}
Let \(E,F\in\mathrm{Gr}(k,q)\), and suppose
\(\|P_E-P_F\|_{\rm op}\le\varepsilon\). For every orthonormal basis
\(U_F\in\mathbb R^{q\times k}\) of \(F\), there is an orthonormal basis
\(U_E\in\mathbb R^{q\times k}\) of \(E\) such that
\[
    \|U_E-U_F\|_F\le C\sqrt{k}\,\varepsilon .
\]
\end{lemma}

\begin{proof}[Proof of Lemma~\ref{lem:subspace-alignment}.]
The assertion is invariant under simultaneous right multiplication of
\(U_E\) and \(U_F\) by an orthogonal \(k\times k\) matrix, since this does not
change the corresponding subspaces and preserves the Frobenius norm. Hence, we may assume without loss of
generality that \(U_F\) is a principal basis of \(F\).
Let \(\theta_1,\ldots,\theta_k\in[0,\pi/2]\) be the principal angles between
\(E\) and \(F\). Since
\(\|P_E-P_F\|_{\rm op}=\max_{1\le i\le k}\sin\theta_i\), we have
\(\sin\theta_i\le\varepsilon\) for every \(i\). Choose corresponding
principal vector bases \((e_i)_{i=1}^k\) and \((f_i)_{i=1}^k\) of \(E\) and
\(F\). Taking \(U_E=(e_1,\ldots,e_k)\) and \(U_F=(f_1,\ldots,f_k)\), we get
\[
\begin{aligned}
    \|U_E-U_F\|_F^2
    &=
    \sum_{i=1}^k \|e_i-f_i\|_2^2 =
    2\sum_{i=1}^k (1-\cos\theta_i)\le
    2\sum_{i=1}^k \sin^2\theta_i
    \le
    2k\varepsilon^2.
\end{aligned}
\]
This proves the claim, after changing the universal constant.
\end{proof}

\begin{proof}[Proof of Proposition~\ref{prop:microscopic-hard-edge-counting}.]
Whenever this proof requires \(n\) to be large, the lower bound on \(n\) is
chosen uniformly over \(1\le k\le q_n\) and \(0<t\le1\).
For this proof, write \(q:=q_n\), \(N:=N_n\), \(r:=r_n=N-q\), and
\(A_n:=n^{-1}G_n^\top G_n\). Let
\(0<s_1(G_n)\le \cdots \le s_q(G_n)\) be the singular values of \(G_n\). Since
the eigenvalues of \(A_n\) are \(s_i(G_n)^2/n\), the event
\(\{\lambda_{k,n}\le t\}\) is the same as \(\{s_k(G_n)^2\le nt\}\).
By the min--max characterization of singular values,
\[
    s_k(G)
    =
    \inf_{\substack{E\subset\mathbb R^q\\ \dim(E)=k}}
    \|G|_E\|_{\mathrm{op}}.
\]
Hence, on the event \(\{\lambda_k\le t\}\), there exists a
\(k\)-dimensional subspace \(E\subset\mathbb R^q\) such that
\(\|G|_E\|_{\mathrm{op}}\le s=\sqrt{nt}\).
If \(U_E\in\mathbb R^{q\times k}\) is an orthonormal basis of \(E\), then
\(\|GU_E\|_{\mathrm{op}}\le\sqrt{nt}\), and therefore
\(\|GU_E\|_F\le\sqrt{k}\,\|GU_E\|_{\mathrm{op}}\le\sqrt{knt}\).

We first restrict to a high-probability operator-norm event.
By the operator-norm bound for
rectangular subgaussian matrices~\cite[Theorem~4.4.3]{vershynin2020high}, there exists a constant
\(C_0<\infty\), depending only on the subgaussian norm of the entries,
such that for every \(u>0\),
\[
    \mathbb P\left(
        \|G\|_{\mathrm{op}}>
        C_0(\sqrt N+\sqrt q+u)
    \right)
    \le
    2e^{-u^2}.
\]
By the aspect-ratio assumption, there exists \(C_{\rm asp}<\infty\) such that,
for all sufficiently large \(n\),
\[
    \sqrt N+\sqrt q\le C_{\rm asp}\sqrt n,
    \qquad
    N\le C_{\rm asp}n,
    \qquad
    q\le C_{\rm asp}n.
\]
Since \(L_n\to\infty\), for all sufficiently large \(n\),
\(L_n\ge2C_0C_{\rm asp}\). Set \(u_n:=L_n\sqrt n/(2C_0)\).
Then
\[
    C_0(\sqrt N+\sqrt q+u_n)
    \le
    C_0C_{\rm asp}\sqrt n+\frac12L_n\sqrt n
    \le
    L_n\sqrt n .
\]
Thus the good event \(\mathcal E_n(L_n):=\{\|G_n\|_{\rm op}\le L_n\sqrt n\}\)
satisfies
\[
    \mathbb P(\mathcal E_n(L_n)^c)
    \le
    2\exp\left\{-\frac{nL_n^2}{4C_0^2}\right\}
    \le
    2\exp\{-c nL_n^2\}.
\]

Let \(C_{\rm align}\) be the universal constant in
Lemma~\ref{lem:subspace-alignment}, and set
\(\varepsilon_n:=\sqrt t/(4C_{\rm align}L_n)\).
Since \(0<t\le1\) and \(L_n\to\infty\), for all sufficiently large \(n\),
\(\varepsilon_n<1\), uniformly in \(t\). Let
\(\mathcal N_{\varepsilon_n}\) be an \(\varepsilon_n\)-net of
\(\mathrm{Gr}(k,q)\) in the projection metric.

Suppose that \(\lambda_{k,n}\le t\) and \(\mathcal E_n(L_n)\) both occur. Choose
\(E\in\mathrm{Gr}(k,q)\) and \(U_E\) as above. Then choose
\(F\in\mathcal N_{\varepsilon_n}\) such that
\(\|P_E-P_F\|_{\rm op}\le \varepsilon_n\).
Fix an orthonormal basis \(U_F\) of \(F\). By
Lemma~\ref{lem:subspace-alignment}, after changing the basis \(U_E\) inside
\(E\), we may assume
\(\|U_E-U_F\|_F\le C_{\rm align}\sqrt{k}\,\varepsilon_n\).
Therefore, on \(\mathcal E_n(L_n)\),
\[
\begin{aligned}
    \|G_nU_F\|_F
    \le
    \|G_nU_E\|_F
    +
    \|G_n(U_F-U_E)\|_F  &\le
    \sqrt{knt}
    +
    \|G_n\|_{\rm op}\|U_F-U_E\|_F \\
    &\le
    \sqrt{knt}
    +
    L_n\sqrt n\,
    C_{\rm align}\sqrt{k}\,\varepsilon_n =
    \sqrt{knt}
    +
    \frac14\sqrt{knt}
    \le
    2\sqrt{knt}.
\end{aligned}
\]
Thus
\[
    \{\lambda_{k,n}\le t\}\cap\mathcal E_n(L_n)
    \subseteq
    \bigcup_{F\in\mathcal N_{\varepsilon_n}}
    \left\{
        \|G_nU_F\|_F\le 2\sqrt{knt}
    \right\}.
\]
Consequently,
\[
    \mathbb P(\lambda_{k,n}\le t,\mathcal E_n(L_n))
    \le
    |\mathcal N_{\varepsilon_n}|
    \sup_{F\in\mathrm{Gr}(k,q)}
    \mathbb P\left(\|G_nU_F\|_F\le 2\sqrt{knt}\right).
\]

We now estimate the fixed-subspace probability. Fix
\(F\in\mathrm{Gr}(k,q)\) and an orthonormal basis
\(U_F\in\mathbb R^{q\times k}\). Vectorize \(G_n\) as
$
    X:=\operatorname{vec}(G_n)\in\mathbb R^{Nq}.
$
Set
$
    L_F:=U_F^\top\otimes I_N
$, then
$
    \operatorname{vec}(G_nU_F)
    =
    (U_F^\top\otimes I_N)\operatorname{vec}(G_n).
$

Since the columns of \(U_F\) are orthonormal,
\[
    L_F L_F^\top
    =
    (U_F^\top\otimes I_N)(U_F\otimes I_N)
    =
    (U_F^\top U_F)\otimes I_N
    =
    I_k\otimes I_N
    =
    I_{Nk}.
\]
Thus the rows of \(L_F\) are orthonormal. Hence
$
    P_F:=L_F^\top L_F
$
is the orthogonal projection in \(\mathbb R^{Nq}\) onto the row space of
\(L_F\), which has dimension \(Nk\).

Moreover, for every
\(x\in\mathbb R^{Nq}\),
$
    \|L_Fx\|_2^2
    =
    x^\top L_F^\top L_Fx
    =
    \|P_Fx\|_2^2
$,
thus the event
$
    \|\operatorname{vec}(G_nU_F)\|_2\le 2\sqrt{knt}
$
is equivalent to
$
    \|P_F\operatorname{vec}(G_n)\|_2\le 2\sqrt{knt}.
$
Since the coordinates of \(\operatorname{vec}(G_n)\)
are independent and have densities bounded by \(K_{\rm up}\),
Proposition~\ref{cor:distance-subspace-upper-small-ball}, applied in
\(\mathbb R^{Nq}\) to
\(H=\operatorname{range}(P_F)^\perp\), whose codimension is \(Nk\), with
radius \(2\sqrt{knt}\), gives
\[
\begin{aligned}
    \mathbb P\left(\|G_nU_F\|_F\le 2\sqrt{knt}\right)
    &=
    \mathbb P\left(
        \|\operatorname{vec}(G_nU_F)\|_2\le 2\sqrt{knt}
    \right) \le
    \left(
        \frac{C K_{\rm up}\sqrt{knt}}{\sqrt{Nk}}
    \right)^{Nk}  =
    \left(
        C K_{\rm up}\sqrt{\frac nN}\sqrt t
    \right)^{Nk}.
\end{aligned}
\]
Since \(N/n\) is bounded below by a positive constant, this becomes
\(\mathbb P(\|G_nU_F\|_F\le 2\sqrt{knt})\le(C\sqrt t)^{Nk}\),
after enlarging \(C\).

Next, by Lemma~\ref{lem:grassmannian-covering},
\[
    |\mathcal N_{\varepsilon_n}|
    \le
    \left(\frac{C}{\varepsilon_n}\right)^{k(q-k)}
    =
    \left(\frac{C L_n}{\sqrt t}\right)^{k(q-k)}.
\]
Combining the net bound and the fixed-subspace small-ball bound gives
\[
\begin{aligned}
    \mathbb P(\lambda_{k,n}\le t,\mathcal E_n(L_n))
    &\le
    \left(\frac{C L_n}{\sqrt t}\right)^{k(q-k)}
    (C\sqrt t)^{Nk} =
    C^{Nk+k(q-k)}
    L_n^{k(q-k)}
    t^{\frac12(Nk-k(q-k))}.
\end{aligned}
\]
The exponent of \(t\) simplifies as
\[
    Nk-k(q-k)
    =
    k(N-q+k)
    =
    k(r+k).
\]
Therefore
\[
    \mathbb P(\lambda_{k,n}\le t,\mathcal E_n(L_n))
    \le
    C^{Nk+k(q-k)}
    L_n^{k(q-k)}
    t^{k(r+k)/2}.
\]

It remains to absorb the prefactor. This step uses only the aspect-ratio
bounds, not any fixed value of \(k\). Since \(N,q=O(n)\),
\(Nk+k(q-k)\le Cnk\), so \(C^{Nk+k(q-k)}\le\exp\{Cnk\}\). Moreover, since
\(q\le Cn\), \(k(q-k)\log L_n\le Cnk\log L_n\). Hence
\(C^{Nk+k(q-k)}L_n^{k(q-k)}\le\exp\{Cnk(1+\log L_n)\}\).
We have proved
\[
    \mathbb P(\lambda_{k,n}\le t,\mathcal E_n(L_n))
    \le
    \exp\{Cnk(1+\log L_n)\}t^{k(r+k)/2}.
\]
Finally,
\[
\begin{aligned}
    \mathbb P(\lambda_{k,n}\le t)
    &\le
    \mathbb P(\lambda_{k,n}\le t,\mathcal E_n(L_n))
    +
    \mathbb P(\mathcal E_n(L_n)^c) \\
    &\le
    \exp\{Cnk(1+\log L_n)\}t^{k(r+k)/2}
    +
    2\exp\{-c nL_n^2\}.
\end{aligned}
\]

To prove the final statement, fix \(M>0\). Choose \(D_M>0\) sufficiently large
and set \(L_n=D_M\sqrt{\log n}\).
Then
\[
    1+\log L_n
    =
    1+\log D_M+\frac12\log\log n
    \le
    C_M\log\log n
\]
for all sufficiently large \(n\), while
\[
    2\exp\{-c nL_n^2\}
    =
    2\exp\{-cD_M^2 n\log n\}
    \le
    \exp\{-M n\log n\}
\]
provided \(D_M\) is chosen so that \(cD_M^2>M+1\).
This proves the proposition.
\end{proof}

\subsection{Left-tail auxiliary estimates}

\begin{lemma}[Continuity of the left-tail variational rate]
\label{lem:left-variational-rate-continuity}
The function \(s\mapsto \mathcal J_{\kappa,\rho}(s)\) is non-increasing and continuous
on \((0,\infty)\). Moreover,
\begin{equation}
\label{eq:left-strict-sublevel-rate}
    \mathcal J_{\kappa,\rho}(s)
    =
    \inf\{I_{\kappa,\rho}(\mu):T(\mu)<s\}.
\end{equation}
\end{lemma}

\begin{proof}[Proof of Lemma~\ref{lem:left-variational-rate-continuity}.]
Monotonicity is immediate because the constraint set \(\{T\le s\}\) increases
with \(s\).

We first prove right-continuity. Let \(s_m\downarrow s\). Since
\(\mathcal J_{\kappa,\rho}(s_m)\le \mathcal J_{\kappa,\rho}(s)\), it suffices to show
\[
    \mathcal J_{\kappa,\rho}(s)
    \le
    \liminf_{m\to\infty}\mathcal J_{\kappa,\rho}(s_m).
\]
If the right-hand side is \(+\infty\), there is nothing to prove. Otherwise,
choose \(\mu_m\) such that \(T(\mu_m)\le s_m\) and
\(I_{\kappa,\rho}(\mu_m)\le\mathcal J_{\kappa,\rho}(s_m)+1/m\).
Since \(I_{\kappa,\rho}\) is good, after passing to a subsequence we may assume
\(\mu_m\Rightarrow\mu\). Since \(T\) is the supremum of the bounded continuous
truncations of \(u^{-1}\), it is lower semicontinuous, and therefore
\(T(\mu)\le\liminf_{m\to\infty}T(\mu_m)\le s\),
and by lower semicontinuity of \(I_{\kappa,\rho}\),
\(I_{\kappa,\rho}(\mu)\le
\liminf_{m\to\infty}I_{\kappa,\rho}(\mu_m)\).
Therefore
\[
    \mathcal J_{\kappa,\rho}(s)
    \le
    I_{\kappa,\rho}(\mu)
    \le
    \liminf_{m\to\infty}\mathcal J_{\kappa,\rho}(s_m),
\]
which proves right-continuity.

Next we prove the strict-sublevel representation
\eqref{eq:left-strict-sublevel-rate}. The inequality ``\(\le\)'' is trivial.
For the reverse inequality, fix \(\mu\) with \(T(\mu)\le s\) and
\(I_{\kappa,\rho}(\mu)<+\infty\). For \(c>1\), let
\(\mu^{(c)}:=(u\mapsto cu)_\#\mu\). Then
\(T(\mu^{(c)})=c^{-1}T(\mu)<s\). Also,
\(\int u\,d\mu^{(c)}(u)=c\int u\,d\mu(u)\) and
\(\int \log u\,d\mu^{(c)}(u)=\log c+\int\log u\,d\mu(u)\), while
\[
    \iint\log|u-v|\,d\mu^{(c)}(u)d\mu^{(c)}(v)
    =
    \log c+
    \iint\log|u-v|\,d\mu(u)d\mu(v).
\]
Hence \(I_{\kappa,\rho}(\mu^{(c)})\to I_{\kappa,\rho}(\mu)\) as \(c\downarrow1\). Thus any
admissible measure for \(\{T\le s\}\) can be approximated in rate by measures
satisfying \(T<s\), proving \eqref{eq:left-strict-sublevel-rate}.

Finally, let \(s_m\uparrow s\). Then
\(\bigcup_{m\ge1}\{\mu:T(\mu)\le s_m\}=\{\mu:T(\mu)<s\}\).
Since the sets on the left are increasing,
\[
    \lim_{m\to\infty}\mathcal J_{\kappa,\rho}(s_m)
    =
    \inf\{I_{\kappa,\rho}(\mu):T(\mu)<s\}
    =
    \mathcal J_{\kappa,\rho}(s),
\]
where the last equality uses \eqref{eq:left-strict-sublevel-rate}. This proves
left-continuity and completes the proof.
\end{proof}

The proof of Proposition~\ref{prop:left-tail-lower} uses the following
approximation input to reduce the variational problem to compactly supported
piecewise-uniform measures.

\begin{lemma}[Piecewise-uniform approximation]
\label{lem:piecewise-uniform-approximation-left}
Let \(s>0\), and let \(\mu\in\mathcal P(\mathbb R_+)\) satisfy
\(T(\mu)<s\) and \(I_{\kappa,\rho}(\mu)<+\infty\). Then, for every
\(\varepsilon>0\), there exists a compactly supported piecewise-uniform
\(\nu\in\mathcal P(\mathbb R_+)\) such that
\[
    T(\nu)<s,
    \qquad
    I_{\kappa,\rho}(\nu)\le I_{\kappa,\rho}(\mu)+\varepsilon .
\]
\end{lemma}

\begin{proof}[Proof of Lemma~\ref{lem:piecewise-uniform-approximation-left}.]
Here, piecewise uniform means that
\(d\nu(u)=\sum_{\ell=1}^L c_\ell\mathbf 1_{J_\ell}(u)\,du\) for disjoint
intervals \(J_1,\ldots,J_L\subset[m,M]\subset(0,\infty)\) and constants
\(c_1,\ldots,c_L\ge0\). Write
\(V_{\kappa,\rho}(u):=u-\rho\log u\) and
\(\mathcal L(\eta):=\iint \log|u-v|\,d\eta(u)d\eta(v)\).
Then
\[
    I_{\kappa,\rho}(\eta)
    =
    \frac{\kappa}{2}\int V_{\kappa,\rho}\,d\eta
    -
    \frac{\kappa^2}{2}\mathcal L(\eta)
    -C_{\kappa,\rho} .
\]
Since \(I_{\kappa,\rho}(\mu)<+\infty\), both \(\int V_{\kappa,\rho}\,d\mu\) and
\(\mathcal L(\mu)\) are finite in the sense required by the rate function. In
particular, \(\mu\) has no atoms, because any atom would make
\(\mathcal L(\mu)=-\infty\).

\paragraph{Truncation away from zero and infinity.} For
\(0<m<M<\infty\), set \(K_{m,M}:=[m,M]\) and
\(\theta_{m,M}:=\mu(K_{m,M})\); whenever \(\theta_{m,M}>0\), set
\(\mu_{m,M}:=\mu|_{K_{m,M}}/\theta_{m,M}\). As \(m\downarrow0\) and
\(M\uparrow\infty\), \(\theta_{m,M}\to1\) and
\(\mu_{m,M}\Rightarrow\mu\). Moreover, since \(T(\mu)<\infty\),
\(T(\mu_{m,M})\to T(\mu)\), and
\(\int V_{\kappa,\rho}\,d\mu_{m,M}\to
\int V_{\kappa,\rho}\,d\mu\).
It remains to justify convergence of the logarithmic energy. Since
    $\log^+|u-v|
    \le
    \log(1+u)+\log(1+v)
    \le
    u+v$,
the positive part of the logarithmic kernel is integrable under
\(\mu\otimes\mu\). The negative part is integrable because
\(\mathcal L(\mu)>-\infty\). Hence \(\log|u-v|\in L^1(\mu\otimes\mu)\). By
absolute continuity of the integral,
\[
    \iint_{K_{m,M}^2}\log|u-v|\,d\mu(u)d\mu(v)
    \to
    \mathcal L(\mu).
\]
After normalization by \(\theta_{m,M}^{-2}\), this gives
\(\mathcal L(\mu_{m,M})\to\mathcal L(\mu)\), and consequently
\(I_{\kappa,\rho}(\mu_{m,M})\to I_{\kappa,\rho}(\mu)\). Because \(T(\mu)<s\),
we can choose \(m,M\) so that \(T(\mu_{m,M})<s\) and
\(I_{\kappa,\rho}(\mu_{m,M})\le I_{\kappa,\rho}(\mu)+\varepsilon/3\).
Fix such \(m,M\), and write \(\bar\mu:=\mu_{m,M}\).

\paragraph{Smoothing on a compact interval.} Let \(\varphi\) be a smooth
probability density supported on \([-1,1]\), and put
\(\varphi_h(u):=h^{-1}\varphi(u/h)\), where \(0<h<m/4\).
Define \(\bar\mu_h:=\bar\mu*\varphi_h\). Then \(\bar\mu_h\) has a smooth
density and \(\operatorname{supp}(\bar\mu_h)\subset[m-h,M+h]\subset[m/2,2M]\).
Since \(u\mapsto u^{-1}\) and \(V_{\kappa,\rho}\) are bounded and uniformly continuous
on \([m/2,2M]\), \(T(\bar\mu_h)\to T(\bar\mu)\) and
\(\int V_{\kappa,\rho}\,d\bar\mu_h\to\int V_{\kappa,\rho}\,d\bar\mu\). We
claim that \(\mathcal L(\bar\mu_h)\to\mathcal L(\bar\mu)\).
Let \(X,Y\) be independent with law \(\bar\mu\), and let \(U,V\) be independent
with density \(\varphi\), independent of \(X,Y\). Then
\(\mathcal L(\bar\mu_h)=\mathbb E\log|X-Y+h(U-V)|\).
For fixed \(X\ne Y\), the integrand converges to \(\log|X-Y|\). Moreover,
since \(\bar\mu\) has finite logarithmic energy and no atoms, the singular part
is uniformly integrable. More explicitly, for \(h\le m/4\),
\[
    \mathbb E_{U,V}\bigl[-\log|z+h(U-V)|\bigr]^+
    \le
    C+[-\log|z|]^+,
    \qquad z\in\mathbb R,
\]
where \(C\) depends only on \(\varphi\). The right-hand side is integrable with
respect to the law of \(X-Y\), by the finite logarithmic energy of \(\bar\mu\).
The positive part of the logarithm is uniformly bounded on \([m/2,2M]^2\). Thus
dominated convergence gives \(\mathcal L(\bar\mu_h)\to\mathcal L(\bar\mu)\).
Hence \(I_{\kappa,\rho}(\bar\mu_h)\to I_{\kappa,\rho}(\bar\mu)\). Choosing
\(h>0\) sufficiently small, \(T(\bar\mu_h)<s\) and
\(I_{\kappa,\rho}(\bar\mu_h)\le I_{\kappa,\rho}(\bar\mu)+\varepsilon/3\).

\paragraph{Approximation by a piecewise-constant density.} Let \(f_h\) be
the smooth density of \(\bar\mu_h\), supported on \([m/2,2M]\). Partition
\([m/2,2M]\) into finitely many intervals \(J_1,\ldots,J_L\) with mesh size
tending to zero. Let \(f_L\) be the cell-average approximation
\[
    f_L(u)
    :=
    \sum_{\ell=1}^L
    \left(
        \frac1{|J_\ell|}\int_{J_\ell}f_h(v)\,dv
    \right)\mathbf 1_{J_\ell}(u),
\]
and define \(d\nu_L(u):=f_L(u)\,du\). Then \(\nu_L\) is piecewise uniform,
\(\nu_L\Rightarrow\bar\mu_h\), and \(\|f_L-f_h\|_{L^1}\to0\).
Since \(u^{-1}\) and \(V_{\kappa,\rho}\) are bounded continuous on \([m/2,2M]\),
\(T(\nu_L)\to T(\bar\mu_h)\) and
\(\int V_{\kappa,\rho}\,d\nu_L\to\int V_{\kappa,\rho}\,d\bar\mu_h\).
For the logarithmic energy, note that the densities \(f_L\) are uniformly
bounded for all sufficiently large \(L\), because \(f_h\) is bounded. Fix
\(A>0\) and define the truncated continuous kernel
\(K_A(u,v):=\max\{\log|u-v|,-A\}\).
Then
\[
    \iint K_A(u,v)\,d\nu_L(u)d\nu_L(v)
    \to
    \iint K_A(u,v)\,d\bar\mu_h(u)d\bar\mu_h(v).
\]
The error between \(K_A\) and \(\log|u-v|\) is supported on
\(|u-v|\le e^{-A}\). Since all densities are uniformly bounded,
\[
    \sup_L
    \iint_{|u-v|\le e^{-A}}
    |\log|u-v||\,d\nu_L(u)d\nu_L(v)
    \le
    C\iint_{|u-v|\le e^{-A}}|\log|u-v||\,du\,dv
    \to0
\]
as \(A\to\infty\). The same bound holds for \(\bar\mu_h\). Therefore
\(\mathcal L(\nu_L)\to\mathcal L(\bar\mu_h)\). Thus
\(I_{\kappa,\rho}(\nu_L)\to I_{\kappa,\rho}(\bar\mu_h)\).
Taking \(L\) sufficiently large, we obtain a piecewise-uniform measure
\(\nu:=\nu_L\) such that
\[
    T(\nu)<s,
    \qquad
    I_{\kappa,\rho}(\nu)
    \le
    I_{\kappa,\rho}(\bar\mu_h)+\frac{\varepsilon}{3}
    \le
    I_{\kappa,\rho}(\mu)+\varepsilon.
\]
This proves the lemma.
\end{proof}

\bibliographystyle{plainnat}
\bibliography{ref}

@article{hastie2022surprises,
  title={Surprises in high-dimensional ridgeless least squares interpolation},
  author={Hastie, Trevor and Montanari, Andrea and Rosset, Saharon and Tibshirani, Ryan J},
  journal={The Annals of Statistics},
  volume={50},
  number={2},
  pages={949--986},
  year={2022}
}

@article{vaaler1979geometric,
  title={A geometric inequality with applications to linear forms},
  author={Vaaler, Jeffrey D},
  journal={Pacific Journal of Mathematics},
  volume={83},
  number={2},
  pages={543--553},
  year={1979},
  publisher={Mathematical Sciences Publishers}
}

@article{rogers1958packing,
  title={The packing of equal spheres},
  author={Rogers, C. A.},
  journal={Proceedings of the London Mathematical Society},
  volume={s3-8},
  number={4},
  pages={609--620},
  year={1958},
  publisher={Oxford University Press}
}

@inproceedings{varadhan2010large,
  title={Large deviations},
  author={Varadhan, S. R. S.},
  booktitle={Proceedings of the International Congress of Mathematicians},
  editor={Bhatia, Rajendra},
  volume={I: Plenary lectures and ceremonies},
  pages={622--639},
  year={2010},
  publisher={Hindustan Book Agency},
  address={New Delhi}
}

@book{dembo2009large,
  title={Large Deviations Techniques and Applications},
  author={Dembo, Amir and Zeitouni, Ofer},
  series={Stochastic Modelling and Applied Probability},
  volume={38},
  edition={2nd},
  year={2010},
  publisher={Springer Berlin Heidelberg}
}

@article{artzner1999coherent,
  title={Coherent measures of risk},
  author={Artzner, Philippe and Delbaen, Freddy and Eber, Jean-Marc and Heath, David},
  journal={Mathematical finance},
  volume={9},
  number={3},
  pages={203--228},
  year={1999},
  publisher={Wiley Online Library}
}

@article{duffie1997overview,
  title={An overview of value at risk},
  author={Duffie, Darrell and Pan, Jun},
  journal={The Journal of Derivatives},
  volume={4},
  number={3},
  pages={7--49},
  year={1997}
}

@article{zhang2021understanding,
  title={Understanding deep learning (still) requires rethinking generalization},
  author={Zhang, Chiyuan and Bengio, Samy and Hardt, Moritz and Recht, Benjamin and Vinyals, Oriol},
  journal={Communications of the ACM},
  volume={64},
  number={3},
  pages={107--115},
  year={2021},
  publisher={ACM New York, NY, USA}
}

@article{mallinar2022benign,
  title={Benign, tempered, or catastrophic: Toward a refined taxonomy of overfitting},
  author={Mallinar, Neil and Simon, James and Abedsoltan, Amirhesam and Pandit, Parthe and Belkin, Misha and Nakkiran, Preetum},
  journal={Advances in neural information processing systems},
  volume={35},
  pages={1182--1195},
  year={2022}
}

@inproceedings{cheng2022memorize,
  title={Memorize to generalize: on the necessity of interpolation in high dimensional linear regression},
  author={Cheng, Chen and Duchi, John and Kuditipudi, Rohith},
  booktitle={Conference on Learning Theory},
  pages={5528--5560},
  year={2022},
  organization={PMLR}
}

@article{mei2022generalization,
  title={The generalization error of random features regression: Precise asymptotics and the double descent curve},
  author={Mei, Song and Montanari, Andrea},
  journal={Communications on Pure and Applied Mathematics},
  volume={75},
  number={4},
  pages={667--766},
  year={2022},
  publisher={Wiley Online Library}
}

@article{bartlett2020benign,
  title={Benign overfitting in linear regression},
  author={Bartlett, Peter L and Long, Philip M and Lugosi, G{\'a}bor and Tsigler, Alexander},
  journal={Proceedings of the National Academy of Sciences},
  volume={117},
  number={48},
  pages={30063--30070},
  year={2020},
  publisher={National Academy of Sciences}
}

@article{cao2022benign,
  title={Benign overfitting in two-layer convolutional neural networks},
  author={Cao, Yuan and Chen, Zixiang and Belkin, Misha and Gu, Quanquan},
  journal={Advances in neural information processing systems},
  volume={35},
  pages={25237--25250},
  year={2022}
}

@article{belkin2021fit,
  title={Fit without fear: remarkable mathematical phenomena of deep learning through the prism of interpolation},
  author={Belkin, Mikhail},
  journal={Acta Numerica},
  volume={30},
  pages={203--248},
  year={2021},
  publisher={Cambridge University Press}
}

@article{belkin2019reconciling,
  title={Reconciling modern machine-learning practice and the classical bias--variance trade-off},
  author={Belkin, Mikhail and Hsu, Daniel and Ma, Siyuan and Mandal, Soumik},
  journal={Proceedings of the National Academy of Sciences},
  volume={116},
  number={32},
  pages={15849--15854},
  year={2019},
  publisher={National Academy of Sciences}
}

@article{belkin2018overfitting,
  title={Overfitting or perfect fitting? risk bounds for classification and regression rules that interpolate},
  author={Belkin, Mikhail and Hsu, Daniel J and Mitra, Partha P},
  journal={Advances in neural information processing systems},
  volume={31},
  pages={2306--2317},
  year={2018}
}

@article{rudelson2009smallest,
  title={The smallest singular value of a random rectangular matrix},
  author={Rudelson, Mark and Vershynin, Roman},
  journal={Communications on Pure and Applied Mathematics: A Journal Issued by the Courant Institute of Mathematical Sciences},
  volume={62},
  number={12},
  pages={1707--1739},
  year={2009},
  publisher={Wiley Online Library}
}

@book{vershynin2020high,
  title={High-Dimensional Probability: An Introduction with Applications in Data Science},
  author={Vershynin, Roman},
  series={Cambridge Series in Statistical and Probabilistic Mathematics},
  volume={58},
  edition={2nd},
  year={2026},
  publisher={Cambridge University Press}
}

@article{nguyen2018random,
  title={Random matrices: Overcrowding estimates for the spectrum},
  author={Nguyen, Hoi H},
  journal={Journal of functional analysis},
  volume={275},
  number={8},
  pages={2197--2224},
  year={2018},
  publisher={Elsevier}
}

@article{dobriban2018high,
  title={High-dimensional asymptotics of prediction: Ridge regression and classification},
  author={Dobriban, Edgar and Wager, Stefan},
  journal={The Annals of Statistics},
  volume={46},
  number={1},
  pages={247--279},
  year={2018},
  publisher={JSTOR}
}

@book{hiai2000semicircle,
  title={The semicircle law, free random variables and entropy},
  author={Hiai, Fumio and Petz, D{\'e}nes},
  series={Mathematical Surveys and Monographs},
  volume={77},
  year={2000},
  publisher={American Mathematical Society}
}

@article{rudelson2015small,
  title={Small ball probabilities for linear images of high-dimensional distributions},
  author={Rudelson, Mark and Vershynin, Roman},
  journal={International Mathematics Research Notices},
  volume={2015},
  number={19},
  pages={9594--9617},
  year={2015},
  publisher={Oxford University Press}
}

@article{marvcenko1967distribution,
  title={Distribution of eigenvalues for some sets of random matrices},
  author={Mar{\v{c}}enko, Vladimir A and Pastur, Leonid Andreevich},
  journal={Mathematics of the USSR-Sbornik},
  volume={1},
  number={4},
  pages={457--483},
  year={1967}
}

@book{bai2010spectral,
  title={Spectral analysis of large dimensional random matrices},
  author={Bai, Zhidong and Silverstein, Jack W},
  series={Springer Series in Statistics},
  edition={2nd},
  year={2010},
  publisher={Springer}
}

@article{pillai2014universality,
  title={Universality of covariance matrices},
  author={Pillai, Natesh S and Yin, Jun},
  journal={The Annals of Applied Probability},
  volume={24},
  number={3},
  pages={935--1001},
  year={2014},
  publisher={JSTOR}
}

@article{james1964distributions,
  title={Distributions of matrix variates and latent roots derived from normal samples},
  author={James, Alan T},
  journal={The Annals of Mathematical Statistics},
  volume={35},
  number={2},
  pages={475--501},
  year={1964},
  publisher={Institute of Mathematical Statistics}
}

@article{hiai1998eigenvalue,
  title={Eigenvalue density of the Wishart matrix and large deviations},
  author={Hiai, Fumio and Petz, D{\'e}nes},
  journal={Infinite Dimensional Analysis, Quantum Probability and Related Topics},
  volume={1},
  number={4},
  pages={633--646},
  year={1998},
  publisher={World Scientific}
}

@book{forrester2010log,
  title={Log-Gases and Random Matrices},
  author={Forrester, Peter J},
  series={London Mathematical Society Monographs},
  volume={34},
  year={2010},
  publisher={Princeton University Press}
}

@article{ledoux2006concentration,
  title={Concentration of measure and logarithmic Sobolev inequalities},
  author={Ledoux, Michel},
  journal={S{\'e}minaire de probabilit{\'e}s},
  volume={33},
  pages={120--216},
  year={1999},
  publisher={Springer}
}

@incollection{pajor1998metric,
  title={Metric entropy of the Grassmann manifold},
  author={Pajor, Alain},
  booktitle={Convex Geometric Analysis},
  editor={Ball, Keith M and Milman, Vitali D},
  series={Mathematical Sciences Research Institute Publications},
  volume={34},
  pages={181--188},
  year={1998},
  publisher={Cambridge University Press}
}

@article{han2026distribution,
  title={The distribution of ridgeless least squares interpolators},
  author={Han, Qiyang and Xu, Xiaocong},
  journal={Journal of Machine Learning Research},
  volume={27},
  number={23},
  pages={1--94},
  year={2026}
}

@article{smith2006optimizer,
  title={The optimizer’s curse: Skepticism and postdecision surprise in decision analysis},
  author={Smith, James E and Winkler, Robert L},
  journal={Management Science},
  volume={52},
  number={3},
  pages={311--322},
  year={2006},
  publisher={INFORMS}
}

@article{bertsimas2020predictive,
  title={From predictive to prescriptive analytics},
  author={Bertsimas, Dimitris and Kallus, Nathan},
  journal={Management Science},
  volume={66},
  number={3},
  pages={1025--1044},
  year={2020},
  publisher={INFORMS}
}

@article{elmachtoub2022smart,
  title={Smart “predict, then optimize”},
  author={Elmachtoub, Adam N and Grigas, Paul},
  journal={Management Science},
  volume={68},
  number={1},
  pages={9--26},
  year={2022},
  publisher={INFORMS}
}

@inproceedings{donti2017task,
  title={Task-based end-to-end model learning in stochastic optimization},
  author={Donti, Priya L and Amos, Brandon and Kolter, J. Zico},
  booktitle={Advances in Neural Information Processing Systems},
  volume={30},
  pages={5484--5494},
  year={2017}
}

@article{edelman1988eigenvalues,
  title={Eigenvalues and condition numbers of random matrices},
  author={Edelman, Alan},
  journal={SIAM journal on matrix analysis and applications},
  volume={9},
  number={4},
  pages={543--560},
  year={1988},
  publisher={SIAM}
}

@article{hill1975simple,
  title={A simple general approach to inference about the tail of a distribution},
  author={Hill, Bruce M},
  journal={The Annals of Statistics},
  volume={3},
  number={5},
  pages={1163--1174},
  year={1975},
  publisher={JSTOR}
}

@article{tao2010random,
  title={Random matrices: The distribution of the smallest singular values},
  author={Tao, Terence and Vu, Van},
  journal={Geometric and Functional Analysis},
  volume={20},
  number={1},
  pages={260--297},
  year={2010},
  publisher={Springer}
}

@article{katzav2010large,
  title={Large deviations of the smallest eigenvalue of the Wishart-Laguerre ensemble},
  author={Katzav, Eytan and P{\'e}rez Castillo, Isaac},
  journal={Physical Review E},
  volume={82},
  number={4},
  pages={040104},
  year={2010},
  publisher={APS}
}

@article{matsumoto2012general,
  title={General moments of the inverse real Wishart distribution and orthogonal Weingarten functions},
  author={Matsumoto, Sho},
  journal={Journal of Theoretical Probability},
  volume={25},
  number={3},
  pages={798--822},
  year={2012},
  publisher={Springer}
}

@article{kumari2017moments,
  title={Moments of inverses of $(m,n,\beta)$-{Laguerre} matrices},
  author={Kumari, Sushma},
  journal={arXiv preprint arXiv:1704.06878},
  year={2017}
}

@article{mirsky1957inequalities,
author = {Mirsky, Leon},
title = {Inequalities for normal and {Hermitian} matrices},
journal = {Duke Mathematical Journal},
volume = {24},
number = {4},
pages = {591--599},
year = {1957}
}

@article{li1999lidskii,
author = {Li, Chi-Kwong and Mathias, Roy},
title = {The {Lidskii--Mirsky--Wielandt} theorem: additive and multiplicative versions},
journal = {Numerische Mathematik},
volume = {81},
pages = {377--413},
year = {1999}
}

@article{tao2010random2,
  author  = {Tao, Terence and Vu, Van and Krishnapur, Manjunath},
  title   = {Random matrices: Universality of {ESDs} and the circular law},
  journal = {The Annals of Probability},
  volume  = {38},
  number  = {5},
  pages   = {2023--2065},
  year    = {2010}
}

@article{chen2026abundant,
  title={How abundant are good interpolators?},
  author={Chen, August Y. and {El Alaoui}, Ahmed},
  journal={arXiv preprint arXiv:2606.06469},
  year={2026}
}

\end{document}